\documentclass{amsart}

\usepackage[T1]{fontenc}
\usepackage[latin1]{inputenc}
\usepackage{url}
\usepackage[colorlinks=true]{hyperref}

\newcommand{\dd}{\mathrm{d}}
\newcommand{\R}{\mathbb{R}}
\renewcommand{\P}{\mathcal{P}}
\newcommand{\Exp}{\mathbb{E}}
\renewcommand{\Pr}{\mathbb{P}}
\newcommand{\ind}[1]{\mathbf{1}_{\{#1\}}}
\newcommand{\Supp}{\mathrm{Supp}\,}
\newcommand{\sgn}{\,\mathrm{sgn}}
\newcommand{\Lloc}{L^1_{\mathrm{loc}}(\R)}
\newcommand{\Cc}{C^{\infty}_{\mathrm{c}}}
\newcommand{\tA}{\tilde{A}}
\newcommand{\tB}{\tilde{B}}

\newcommand{\tmu}{\tilde{\mu}}
\newcommand{\tpi}{\tilde{\pi}}
\newcommand{\tP}{\tilde{P}}

\newcommand{\bF}{\bar{F}}
\newcommand{\F}{\mathcal{F}}

\newcommand{\g}{g^{\delta}_{\eta,\epsilon}}
\newcommand{\gk}{g^{\delta}_{\eta,\epsilon_k}}
\newcommand{\ua}{\underline{a}}

\newcommand{\ta}{\mathrm{tl}}
\newcommand{\A}{\mathcal{A}}
\newcommand{\Lip}{\mathrm{b}}
\newcommand{\G}{\mathcal{G}}

\newcommand{\sk}{\vskip 3mm}

\newtheorem{defi}{Definition}[section]
\newtheorem{lem}[defi]{Lemma}
\newtheorem{prop}[defi]{Proposition}
\newtheorem{theo}[defi]{Theorem}
\newtheorem{cor}[defi]{Corollary}

\newtheoremstyle{myremark}{}{}{}{0pt}{\bfseries}{.}{ }{}
\theoremstyle{myremark}
\newtheorem{rk}[defi]{Remark}

\title[Propagation of chaos for rank-based interacting diffusions]{Propagation of chaos for rank-based interacting diffusions and long time behaviour of a scalar quasilinear parabolic equation}

\author{Benjamin Jourdain}
\address{Université Paris-Est, CERMICS\\
6 et 8 avenue Blaise Pascal, Cité Descartes\\
77455 Marne-la-Vallée Cedex 2}
\email{\href{mailto:jourdain@cermics.enpc.fr}{jourdain@cermics.enpc.fr}}

\author{Julien Reygner}
\address{Laboratoire de Probabilités et Modèles Aléatoires (CNRS UMR. 7599)\\
Université Paris 6 --- Pierre et Marie Curie, U.F.R. Mathématiques\\
Case 188, 4 place Jussieu\\
75252 Paris Cedex 05}
\address{Université Paris-Est, CERMICS\\
6 et 8 avenue Blaise Pascal, Cité Descartes\\
77455 Marne-la-Vallée Cedex 2}
\email{\href{mailto:julien.reygner@upmc.fr}{julien.reygner@upmc.fr}}

\thanks{This work is supported by the French National Research Agency under the grant ANR-12-BLAN-Stab. The final publication is available at \url{http://link.springer.com/article/10.1007/s40072-013-0014-2}.}

\begin{document}

\begin{abstract}
  We study a quasilinear parabolic Cauchy problem with a cumulative distribution function on the real line as an initial condition. We call `probabilistic solution' a weak solution which remains a cumulative distribution function at all times. We prove the uniqueness of such a solution and we deduce the existence from a propagation of chaos result on a system of scalar diffusion processes, the interactions of which only depend on their ranking. We then investigate the long time behaviour of the solution. Using a probabilistic argument and under weak assumptions, we show that the flow of the Wasserstein distance between two solutions is contractive. Under more stringent conditions ensuring the regularity of the probabilistic solutions, we finally derive an explicit formula for the time derivative of the flow and we deduce the convergence of solutions to equilibrium.
\end{abstract}

\maketitle

\section*{Introduction}

Let $a,b : [0,1] \to \R$ be continuous functions, with $a \geq 0$. For all $u \in [0,1]$, let us define $A(u) = \int_0^u a(v)\dd v$ and $B(u) = \int_0^u b(v)\dd v$. Let $m$ be a probability distribution on $\R$. We are interested in the nonlinear Cauchy problem on $[0,+\infty) \times \R$:
\begin{equation}\label{eq:myPDE}
  \left\{\begin{aligned}
    & \partial_t F_t(x) = \frac{1}{2}\partial_x^2 \big(A(F_t(x))\big) - \partial_x \big(B(F_t(x))\big),\\
    & F_0(x) = H*m(x),
  \end{aligned}\right.
\end{equation}
where $H * \cdot$ refers to the spatial convolution with the Heaviside function. 

The partial differential equation in~\eqref{eq:myPDE} is called a scalar {\em quasilinear parabolic} equation. It is a model for several usual nonlinear evolution equations, such as the porous medium equation, for which $B(u)=0$ and the diffusion term has the particular form $A(u) = u^q$, $q>1$; or conservation laws, in which the diffusion term is linear, i.e. $A(u) = \sigma^2 u$ with $\sigma^2 \geq 0$. A conservation law is said to be {\em viscous} if $\sigma^2>0$ and {\em inviscid} if $\sigma^2=0$. A particular case of a conservation law is the Burgers equation, for which $B(u)=u^2$.  

In this article, we introduce a probabilistic approximation of the Cauchy problem by means of a system of scalar diffusion processes, interacting through their ranking. We then use this probabilistic representation to study the long time behaviour of the solution.

\sk
A {\em weak solution} to the Cauchy problem~\eqref{eq:myPDE} is a continuous mapping $F : t \in [0,+\infty) \mapsto F_t \in \Lloc$ such that for all $t \geq 0$, $F_t$ takes its values in $[0,1]$ and for all $g \in \Cc([0, +\infty) \times \R)$,
\begin{equation}\label{eq:weak}
  \begin{split}
    & \int_{\R} g(t,x)F_t(x)\dd x - \int_{\R} g(0,x)H*m(x)\dd x\\
    & \qquad =\int_{\R}\int_0^t \left\{\frac{1}{2}A(F_s(x))\partial_x^2g(s,x) + B(F_s(x))\partial_x g(s,x) + F_s(x)\partial_s g(s,x)\right\}\dd s \dd x,
  \end{split}
\end{equation}
where $\Cc([0, +\infty) \times \R)$ refers to the space of real-valued $C^{\infty}$ functions with compact support in $[0, +\infty) \times \R$. 

When $A$ is $C^2$ on $[0,1]$, we say that $F$ has the {\em classical regularity} if $F$ is $C^{1,2}$ on $(0,+\infty) \times \R$ and solves~\eqref{eq:myPDE} in the classical sense. The space derivative $p$ of a solution $F$ with classical regularity satisfies the nonlinear Fokker-Planck equation
\begin{equation}\label{eq:EFP}
  \partial_t p_t(x) = \frac{1}{2} \partial_x^2\big(a(H*p_t(x))p_t(x)\big) - \partial_x\big(b(H*p_t(x))p_t(x)\big),
\end{equation}
therefore it is natural to consider the associated {\em nonlinear} stochastic differential equation
\begin{equation}\label{eq:nldp}
  \left\{\begin{aligned}
    &X_t = X_0 + \int_0^t b(H*P_s(X_s))\dd s + \int_0^t \sigma(H*P_s(X_s))\dd W_s,\\
    &\text{$P_t$ is the distribution of $X_t$},
  \end{aligned}\right.  
\end{equation}
where $\sigma(u):=a(u)^{1/2}$, $X_0$ has distribution $m$ and is independent of the Brownian motion $W$. Due to the discontinuity of the Heaviside function, a direct study of this equation by classical techniques, such as the use of fixed-point theorems~\cite{sznitman}, seems out of reach (except when the diffusion coefficient $\sigma$ is constant, see~\cite{jourdain:burgers}); therefore we introduce a linearized approximation of~\eqref{eq:nldp}.

\sk
In Section~\ref{s:prob}, we call {\em particle system} a solution $X^n \in C([0,+\infty),\R^n)$ to the stochastic differential equation in $\R^n$
\begin{equation}\label{eq:sp}
  \left\{\begin{aligned}
    &X^{i,n}_t = X^i_0 + \int_0^t b(H*\mu^n_s(X_s^{i,n}))\dd s + \int_0^t \left(c_n + \sigma(H*\mu^n_s(X_s^{i,n}))\right)\dd W_s^i,\\
    &\mu^n_t = \frac{1}{n}\sum_{i=1}^n \delta_{X^{i,n}_t},
  \end{aligned}\right.
\end{equation}
where $(X_0^i)_{i \geq 1}$ is a sequence of i.i.d. random variables with marginal distribution $m$, independent of the $\R^n$-valued Brownian motion $(W^1, \ldots, W^n)$. Here, the term $c_n > 0$ has been added in the diffusion in order to ensure the well definition of solutions, and it is only required to vanish when $n \to +\infty$. In Proposition~\ref{prop:sp}, we prove that as soon as the function $A$ is increasing and $m$ has a finite first order moment, the flow of time marginals $t \mapsto \mu_t^n$ of the empirical distribution of the particle system converges in probability to the unique mapping $t \mapsto P(t)$ such that the function $F : (t,x) \mapsto (H*P(t))(x)$ is a weak solution to the Cauchy problem~\eqref{eq:myPDE}. This function will be referred to as the {\em probabilistic solution} of the Cauchy problem. Our analysis is essentially based on results obtained for the particular case of the porous medium equation~\cite{jourdain:porous}. A crucial argument in the extension of this work is the uniqueness of weak solutions to the Cauchy problem stated in Proposition~\ref{prop:unicite-edp}, the proof of which is adaptated from works by Wu, Zhao, Yin and Lin~\cite{zhao-livre} and Liu and Wang~\cite{liu} (see Appendix~\ref{app:unicite}). Proposition~\ref{prop:sp} has strong connections with recent results by Shkolnikov~\cite{shkolnikov}, see Remark~\ref{rk:shkolnikov}.

We then give two representations of the mapping $t \mapsto P(t)$ as the flow of time marginals of a probability distribution on the space of sample-paths $C([0,+\infty),\R)$. More precisely, in Subsection~\ref{ss:nlmp}, we give necessary and sufficient conditions on the coefficients of the Cauchy problem for the empirical distribution $\mu^n = (1/n)\sum_{i=1}^n \delta_{X^{i,n}}$ of the particle system in $C([0,+\infty),\R)$ to converge in probability to the law $P$ of a weak solution to~\eqref{eq:nldp}. In Subsection~\ref{ss:spr}, we define the {\em reordered particle system} as the reflected diffusion process obtained by increasingly reordering the positions $(X^{1,n}_t, \ldots, X^{n,n}_t)$ of the particles. We prove in Proposition~\ref{prop:spr} that the associated empirical distribution $\tmu^n$ converges in probability to a probability distribution $\tP$ on $C([0,+\infty),\R)$ with time marginals $P(t)$. 

\sk
Our motivation for introducing the particle system~\eqref{eq:sp} is the study of nonlinear evolution problems, as it has been done for particular equations by Jourdain~\cite{jourdain:signed, jourdain:porous, jourdain:characteristics}. However, such systems of so-called {\em rank-based interacting particles} also arise in several contexts (see the introduction of~\cite{ips} for references), and have received much attention lately. In particular, motivated by the study of the Atlas model of equity markets introduced by Fernholz~\cite{fernholz} (see also Banner, Fernholz and Karatzas~\cite{banner}), much work has been done about the rank-based stochastic differential equation
\begin{equation}\label{eq:rbps}
  \dd X_t^i = \sum_{j=1}^n \ind{X_t^i=Y_t^j}b_j \dd t + \sum_{j=1}^n \ind{X_t^i=Y_t^j}\sigma_j \dd W_t^i,
\end{equation}  
where $(Y_t^1, \ldots, Y_t^n)$ refers to the increasing reordering of $(X_t^1, \ldots, X_t^n)$. The case $n=2$ is exhaustively studied by Fernholz, Ichiba, Karatzas and Prokaj~\cite{fikp}. For $n \geq 3$, Ichiba, Karatzas and Shkolnikov~\cite{iks} show that strong solutions can be defined as long as there is no triple collision. Triple collisions are studied in~\cite{ik}. Concentration of measure bounds for the local time at collisions and statistics related to this system are given by Pal and Shkolnikov in~\cite{ps}.

As far as the long time behaviour of solutions to~\eqref{eq:rbps} is concerned, Ichiba, Papathanakos, Banner, Karatzas and Fernholz~\cite{ipbkf} prove that under some convexity assumption on the sequence of drift coefficients $(b_j)$, the process of spacings $(Y_t^2-Y_t^1, \ldots, Y_t^n - Y_t^{n-1})$ converges in total variation to its unique stationary distribution when $t \to +\infty$. When the sequence of diffusion coefficients is such that $\sigma^2_2 - \sigma^2_1 = \cdots = \sigma^2_n - \sigma^2_{n-1}$, this stationary distribution is the product of exponential distributions. These results extend the work by Pal and Pitman~\cite{pp}, in which $\sigma_j=1$ for all $j$. In this case, the particle system solution to~\eqref{eq:rbps} does not have any equilibrium, as the process of its center of mass is a drifted Brownian motion. However, the convergence to equilibrium in total variation of its projection on the hyperplane $\{x_1 + \cdots + x_n = 0\}$ can be deduced from the long time behaviour of the process of spacings~\cite{pp}. Convergence rates are provided by Ichiba, Pal and Shkolnikov~\cite{ips} using Lyapounov functionals. Based on the Poincaré inequality satisfied by the stationary distribution, Jourdain and Malrieu~\cite{jm} prove the convergence to equilibrium in $\chi^2$ distance with an exponential rate, which is uniform in $n$. However, due to the lack of scaling property in the dimension $n$ for the $\chi^2$ distance, one cannot deduce from their result the convergence to equilibrium of the probabilistic solution $F_t$ to the Cauchy problem~\eqref{eq:myPDE}, which is the purpose of Sections~\ref{s:contr} and~\ref{s:conv} of this article.

\sk
In many cases, transport metrics, and the Wasserstein distance in particular, are contractive for the flow of solutions to parabolic equations: see von Renesse and Sturm~\cite{sturm} for the linear Fokker-Planck equation, Carrillo, McCann and Villani~\cite{cmcv}, Cattiaux, Guillin and Malrieu~\cite{cgm} and the recent work by Bolley, Gentil and Guillin~\cite{bgg} for the granular media equation and Bolley, Guillin and Malrieu~\cite{bgm} for the kinetic Vlasov-Fokker-Planck equation. We will prove such a contractivity property by a probabilistic argument and without further regularity assumption, and then take advantage of it to state the convergence to equilibrium of the solutions. 

Let us first recall some useful properties of the one-dimensional Wasserstein distance (see Villani~\cite{villani} for a complete introduction). Let $p \geq 1$. For all probability distributions $\mu$ and $\nu$ on $\R$, we define
\begin{equation*}
  W_p(\mu,\nu) := \inf_{(X,Y) \in \Pi(\mu,\nu)} \Exp\left(|X-Y|^p\right)^{1/p},
\end{equation*}
where $\Pi(\mu,\nu)$ refers to the set of random couples $(X,Y)$ such that $X$ has marginal distribution $\mu$ and $Y$ has marginal distribution $\nu$. As soon as both $\mu$ and $\nu$ have a finite moment of order $p$, then $W_p(\mu,\nu) < +\infty$. 

Given a right-continuous nondecreasing function $F$, we define its pseudo-inverse as $F^{-1}(u) := \inf\{x \in \R : F(x) > u\}$. Then it is a remarkable feature of the one-dimensional case that the Wasserstein distance $W_p(\mu,\nu)$ can be expressed in terms of the pseudo-inverses of the cumulative distribution functions $F_{\mu} := H*\mu$ and $F_{\nu} := H*\nu$ as
\begin{equation}\label{eq:WinvCDF}
  W_p^p(\mu,\nu) = \int_0^1 |F^{-1}_{\mu}(u)-F^{-1}_{\nu}(u)|^p\dd u.
\end{equation}
This leads to the following useful expressions: let $(x_1, \ldots, x_n) \in \R^n$, we denote by $(y_1, \ldots, y_n)$ its increasing reordering and by $\mu^n$ its empirical distribution. Then for all probability distribution $\mu$,
\begin{equation}\label{eq:Wmunmu}
  W_p^p(\mu^n,\mu) = \sum_{i=1}^n \int_{(i-1)/n}^{i/n} |y_i - F_{\mu}^{-1}(u)|^p\dd u \leq \sum_{i=1}^n \int_{(i-1)/n}^{i/n} |x_i - F_{\mu}^{-1}(u)|^p\dd u.
\end{equation}
In particular when $\mu$ is the empirical distribution $\mu'^n$ of some vector $(x'_1, \ldots, x'_n)$ with increasing reordering $(y'_1, \ldots, y'_n)$,
\begin{equation}\label{eq:Wmunmun}
  W_p^p(\mu^n,\mu'^n) = \frac{1}{n}\sum_{i=1}^n |y_i-y'_i|^p \leq \frac{1}{n}\sum_{i=1}^n |x_i-x'_i|^p.
\end{equation}
We finally point out the fact that we will indifferently refer to the Wasserstein distance between $\mu$ and $\nu$ as $W_p(\mu,\nu)$ or $W_p(F_{\mu}, F_{\nu})$.

\sk
In Section~\ref{s:contr} we study the evolution of the Wasserstein distance between two probabilistic solutions $F$ and $G$ of the Cauchy problem~\eqref{eq:myPDE} with different initial conditions $F_0$ and $G_0$. We use the contractivity of the reordered particle system in Proposition~\ref{prop:contr} to prove that the flow $t \mapsto W_p(F_t,G_t)$ is nonincreasing. Then we provide an explicit expression of the time derivative of the flow in Proposition~\ref{prop:quant}. Our work is related to results exposed in the review papers by Carrillo and Toscani~\cite{ct} and Carrillo, Di Francesco and Lattanzio~\cite{cdl}; a further review is given in Remark~\ref{rk:carrillo}.

\sk
Section~\ref{s:conv} is dedicated to the convergence to equilibrium of the solutions. We call {\em stationary solution} the cumulative function $F_{\infty}$ of a probability distribution $m_{\infty}$ with a finite first order moment, such that if $F$ is the probabilistic solution of the Cauchy problem~\eqref{eq:myPDE} with $F_0 = F_{\infty}$, then for all $t \geq 0$, $F_t = F_{\infty}$. It is clear from~\eqref{eq:weak} that $F_{\infty}$ is a stationary solution if and only if it solves the {\em stationary equation} $(1/2) \partial^2 (A(F_{\infty})) - \partial (B(F_{\infty})) = 0$, in the sense of distributions.

We solve the stationary equation in Proposition~\ref{prop:stat}, extending the results of Jourdain and Malrieu~\cite{jm} who deal with the viscous conservation law. Using the results of Section~\ref{s:contr} as well as the probabilistic approximation built in Section~\ref{s:prob}, we then prove in Theorem~\ref{theo:equi} that the probabilistic solutions converge to the stationary solutions in Wasserstein distance.

In~\cite{jm}, the solutions of the viscous conservation law are proven to converge exponentially fast to equilibrium in $\chi^2$ distance, under the condition that the initial measure $m$ be close enough to the stationary solution $m_{\infty}$. Theorem~\ref{theo:equi} does not involve such a condition, but the proof does not provide any indication on the rate of convergence of $F_t$ to the equilibrium. As we remark in Subsection~\ref{ss:rate}, one can recover an exponential rate of convergence in (quadratic) Wasserstein distance in the setting of~\cite{jm}.

\subsection*{Notations} Given a separable metric space $S$, we denote by $\P(S)$ the set of Borel probability distributions on $S$, equipped with the topology of weak convergence. The space $C([0,+\infty),S)$ of continuous functions from $[0,+\infty)$ to $S$ is provided with the topology of the uniform convergence on the compact sets of $[0,+\infty)$. Besides, for all probability distribution $\mu \in \P(C([0,+\infty),S))$, the marginal distribution at time $t \geq 0$ is denoted by $\mu_t \in \P(S)$. The {\em canonical application} $\P(C([0,+\infty),S)) \to C([0,+\infty),\P(S))$ associates the distribution $\mu$ with the flow of its time marginals $t \mapsto \mu_t$. Finally, if $f$ is a real-valued bounded function then $||f||_{\infty}$ refers to the supremum of the function $|f|$. 

\subsection*{Assumptions} Our results are valid under various assumptions on the degeneracy of the parabolic equation. Let us introduce the following conditions:
\begin{itemize}
  \item[(D1)] The function $A$ is increasing.
  \item[(D2)] For all $u \in (0,1]$, $a(u) > 0$.
  \item[(D3)] There exists $\ua > 0$ such that, for all $u \in [0,1]$, $a(u) \geq \ua$.
\end{itemize}
Obviously, (D1) is weaker than (D2), which is weaker than (D3). 

We also introduce the two following conditions on the regularity of the coefficients:
\begin{itemize}
  \item[(R1)] The function $a$ is $C^1$ on $[0,1]$.
  \item[(R2)] The function $a$ is $C^2$ on $[0,1]$, the function $b$ is $C^1$ on $[0,1]$ and there exists $\beta>0$ such that the functions $a''$ and $b'$ are $\beta$-Hölder continuous.
\end{itemize}
The condition (R1) is a natural necessary condition for the Cauchy problem~\eqref{eq:myPDE} to admit classical solutions. The stronger condition (R2) will be used in Lemma~\ref{lem:regul} to ensure the existence of classical solutions to the Fokker-Planck equation~\eqref{eq:EFP}.

Finally, the existence and integrability of stationary solutions will depend on the two following equilibrium conditions:
\begin{itemize}
  \item[(E1)] For all $u \in (0,1)$, $B(u) > 0$, $B(1)=0$ and the function $a/2B$ is locally integrable on $(0,1)$.
  \item[(E2)] The function $a/2B$ is such that
  \begin{equation*}
    \int_0^{1/2} \frac{a(u) u}{2|B(u)|}\dd u + \int_{1/2}^1 \frac{a(u)(1-u)}{2|B(u)|}\dd u < +\infty.
  \end{equation*}
\end{itemize}

\section{Probabilistic approximation of the solution}\label{s:prob}

\subsection{Existence and uniqueness of the probabilistic solution}\label{ss:sp} Following~\cite{bass}, for all $n \geq 1$ there exists a unique weak solution $X^n = (X^{1,n}_t, \ldots, X^{n,n}_t)_{t \geq 0}$ to the stochastic differential equation~\eqref{eq:sp}. We call it the {\em particle system} and denote by $\mu^n$ the random variable in $\P(C([0,+\infty),\R))$ defined by $\mu^n := (1/n)\sum_{i=1}^n \delta_{X^{i,n}}$.

Let $T>0$, possibly $T=+\infty$. We denote by $\P_1(T)$ the set of continuous mappings $t \in [0,T) \mapsto P(t) \in \P(\R)$ such that for all $t \in [0,T)$, the probability distribution $P(t)$ has a finite first order moment and the function $t \mapsto \int_{\R} |x|P(t)(\dd x)$ is locally integrable on $[0,T)$. Let $\F(T) := \{ F : (t,x) \mapsto (H*P(t))(x) ; P \in \P_1(T)\}$; note that $\F(T) \subset C([0,T),\Lloc)$. The particular sets $\P_1(+\infty)$ and $\F(+\infty)$ are simply denoted by $\P_1$ and $\F$.

Throughout this article, we will call {\em probabilistic solution} the solution to the Cauchy problem~\eqref{eq:myPDE} given by the following proposition.

\begin{prop}\label{prop:sp}
  Under the nondegeneracy condition (D1) and the assumption that $m$ has a finite first order moment, there exists a unique weak solution $F$ to the Cauchy problem~\eqref{eq:myPDE} in $\F$, and it writes $F: (t,x) \mapsto (H*P(t))(x)$ where the mapping $t \mapsto P(t)$ is the limit in probability, in $C([0,+\infty),\P(\R))$, of the sequence of mappings $t \mapsto \mu^n_t$.
\end{prop}

The proof of Proposition~\ref{prop:sp} relies on Proposition~\ref{prop:unicite-edp} and Lemmas~\ref{lem:tightness} and~\ref{lem:piinfty}. For all $n \geq 1$, let $\pi'^n$ denote the distribution of $\mu^n$ in $\P(C([0,+\infty),\R))$. In Lemma~\ref{lem:tightness}, we prove that the sequence $(\pi'^n)_{n \geq 1}$ is tight. Since the canonical application $\P(C([0,+\infty),\R)) \to C([0,+\infty),\P(\R))$ is continuous, then the sequence $(\pi^n)_{n \geq 1}$ of the distributions of the random mappings $t \mapsto \mu^n_t$ in $C([0,+\infty),\P(\R))$ is tight.  Let $\pi^{\infty}$ be the limit of a converging subsequence, that we still index by $n$ for convenience. Lemma~\ref{lem:piinfty} combined with Proposition~\ref{prop:unicite-edp} proves that $\pi^{\infty}$ concentrates on a single point $P \in \P_1$, which is such that the function $(t,x) \mapsto (H*P(t))(x)$ is a weak solution of~\eqref{eq:myPDE}.

\begin{prop}\label{prop:unicite-edp}
  Assume that the nondegeneracy condition (D1) holds.
  \begin{enumerate}
    \item Let $T>0$, possibly $T = +\infty$. Let $F^1$ and $F^2 \in \F(T)$, such that for all $g \in \Cc([0,T)\times\R)$, $F^1$ and $F^2$ satisfy~\eqref{eq:weak}. Then $F^1=F^2$ in $\F(T)$.
    \item There is at most one weak solution to the Cauchy problem~\eqref{eq:myPDE} in $\F$.
  \end{enumerate}
\end{prop}
\begin{proof}
  The second point of the proposition clearly follows from the first point, and the first point is proved in Appendix~\ref{app:unicite}.
\end{proof}

\begin{lem}\label{lem:tightness}
  The sequence $(\pi'^n)_{n \geq 1}$ is tight.
\end{lem}
\begin{proof}
  Since the distribution of $(X^{1,n}, \ldots, X^{n,n})$ in $C([0,+\infty),\R^n)$ is symmetric, according to Sznitman~\cite[Proposition~2.2, p.~177]{sznitman}, $(\pi'^n)_{n \geq 1}$ is tight if and only if the sequence of the distributions of the variables $X^{1,n} \in C([0,+\infty),\R)$ is tight. This latter fact classically follows from the fact that for all $n \geq 1$, $X_0^{1,n} = X_0^1$ has distribution $m$ on the one hand, and from the Kolmogorov criterion as well as the boundedness of the coefficients $a$ and $b$ and the sequence $(c_n)_{n \geq 1}$ on the other hand.
\end{proof}

\begin{lem}\label{lem:piinfty}
  Under the assumption that $m$ has a finite first order moment, the distribution $\pi^{\infty}$ is concentrated on the set of mappings $P \in \P_1$ such that the function $(t,x) \mapsto (H*P(t))(x)$ is a weak solution to the Cauchy problem~\eqref{eq:myPDE}.
\end{lem}
\begin{proof}
  We first prove that $\pi^{\infty}$ concentrates on $\P_1$. Let $\mu^{\infty}$ be a variable in $C([0,+\infty),\P(\R))$ with distribution $\pi^{\infty}$. We will prove that for all $t \geq 0$,
  \begin{equation*}
    \sup_{s \in [0,t]} \int_{\R} |x|\mu^{\infty}(s)(\dd x) < +\infty \qquad \text{a.s.},
  \end{equation*}
  so that taking $t$ in a countable unbounded subset of $[0,+\infty)$ yields $\mu^{\infty} \in \P_1$ almost surely.
  
  Let $t \geq 0$. For all $M \geq 0$, the function $f_M : \mu \mapsto \sup_{s \in [0,t]} \int_{\R} (|x|\wedge M)\mu(s)(\dd x)$ is continuous and bounded on $C([0,+\infty),\P(\R))$. For fixed $n$,
  \begin{equation*}
    \begin{split}
      \Exp(f_M(\mu^n)) & \leq \frac{1}{n} \sum_{i=1}^n \Exp\left(\sup_{s \in [0,t]}|X_s^{i,n}|\right)\\
      &\leq \int_{\R} |x|m(\dd x) + t||b||_{\infty} + \frac{1}{n} \sum_{i=1}^n\left[\Exp\left(\sup_{s \in [0,t]} \left| \int_0^s \left(c_n+\sigma(H*\mu_r^n(X_r^{i,n}))\right)\dd W_r^i\right|^2\right)\right]^{1/2}\\
      &\leq C,
    \end{split}
  \end{equation*}
  where we have used the Cauchy-Schwarz inequality in the second line and the Doob inequality as well as the fact that $m$ has a finite first order moment in the third line. The constant $C$ depends neither on $M$ nor on $n$. As a consequence, $\liminf_{M \to +\infty} \Exp(f_M(\mu^{\infty})) \leq C$ and by Fatou's lemma,
  \begin{equation*}
    C \geq \Exp\left(\liminf_{M \to +\infty} f_M(\mu^{\infty})\right) \geq \Exp\left(\sup_{s \in [0,t]} \liminf_{M \to +\infty} \int_{\R} (|x|\wedge M)\mu^{\infty}(s)(\dd x)\right).
  \end{equation*}
  By the monotone convergence theorem, 
  \begin{equation*}
    \liminf_{M \to +\infty} \int_{\R} (|x|\wedge M)\mu^{\infty}(s)(\dd x) = \lim_{M \to +\infty} \int_{\R} (|x|\wedge M)\mu^{\infty}(s)(\dd x) = \int_{\R} |x|\mu^{\infty}(s)(\dd x),
  \end{equation*}
  so that $\Exp\left(\sup_{s \in [0,t]} \int_{\R} |x|\mu^{\infty}(s)(\dd x)\right) \leq C$, which yields the expected result.
  
  It now remains to prove that $(t,x) \mapsto (H*\mu^{\infty}(t))(x)$ is almost surely a weak solution of the Cauchy problem~\eqref{eq:myPDE}. The computation is made for the porous medium equation in~\cite[Lemma 1.5]{jourdain:porous} and can be straightforwardly extended; it relies on the uniform continuity of $a$ and $b$ on $[0,1]$ and the fact that $c_n \to 0$.
\end{proof}

\begin{rk}
  Without assuming neither (D1) nor the existence of a first order moment for $m$, one can still prove that the sequence $(\pi^n)_{n \geq 1}$ is tight and that the limit of any converging subsequence concentrates on weak solutions to the Cauchy problem~\eqref{eq:myPDE}, that of course not necessarily belong to $\F$. Thus, the existence of weak solutions holds under very weak assumptions.
\end{rk}

\begin{rk}\label{rk:shkolnikov}
  The law of large numbers for the sequence of mappings $t \mapsto \mu^n_t$ stated in Proposition~\ref{prop:sp} has recently been addressed under more restrictive conditions on the initial condition $m$ and the coefficients $a$ and $b$. In~\cite{shkolnikov}, Shkolnikov studies the particle system~\eqref{eq:sp} with the specific condition that the process of spacings between two particles with consecutive positions in $\R$ be stationary (the description of the stationary distribution is given in~\cite{ipbkf}). Then in the case where $a$ is affine and $b$ is $C^1$, with $b'$ uniformly negative, the sequence of mappings $t \mapsto \mu^n_t$ is proven to converge in probability, in $C([0,+\infty),\P(\R))$, to the unique mapping $t \mapsto P(t)$ such that $(t,x) \mapsto (H*P(t))(x)$ is a weak solution to the Cauchy problem~\eqref{eq:myPDE} for a specified initial condition $F_0$. In place of our Proposition~\ref{prop:unicite-edp}, the author uses Gilding's theorem for uniqueness~\cite[Theorem~4]{gilding} and therefore needs to assume that any weak solution to the Cauchy problem is continuous on $[0,+\infty)\times\R$ when $m$ does not weight points. 
  
  In the more recent article by Dembo, Shkolnikov, Varadhan and Zeitouni~\cite{dsvz}, the stationarity assumption is removed and the continuity of $F$ is obtained as a consequence of mild regularity and nondegeneracy assumptions on the coefficients of the Cauchy problem~\eqref{eq:myPDE}. More precisely, the authors establish a large deviation principle for the sequence $(\pi^n)_{n \geq 1}$, with a rate function that is infinite on the set of mappings $t \mapsto P(t)$ such that the function $(t,x) \mapsto (H*P(t))(x)$ is discontinuous. They also prove that a zero of the rate function is a mapping $t \mapsto P(t)$ such that the function $(t,x) \mapsto (H*P(t))(x)$ is a continuous weak solution to~\eqref{eq:myPDE}, and deduce the law of large numbers as a consequence of Gilding's uniqueness theorem.
  
  Both approaches heavily rely on the continuity of the solution $F$ to~\eqref{eq:myPDE}, as it is a crucial condition to use Gilding's uniqueness theorem. While we address the regularity of $F$ in Lemma~\ref{lem:regul} below, we insist on the fact that our proof of Proposition~\ref{prop:sp} does not require that $F$ be continuous, which allows us to relax the regularity and nondegeneracy assumptions on $m$, $a$ and $b$ with respect to~\cite{shkolnikov,dsvz}. However, the regularity of $F$ plays a more important role in establishing the law of large numbers for the sequence of empirical distributions $\mu^n \in \P(C([0,+\infty),\R))$ in Subsection~\ref{ss:nlmp}. Therefore, in the proof of Lemma~\ref{lem:existnlmp}, we prove that, under the nondegeneracy condition (D2), the function $F_t$ is continuous on $\R$, $\dd t$-almost everywhere.
\end{rk}

We conclude this subsection by discussing the regularity of the probabilistic solution $F$. For all finite $T>0$, we denote by $C^{1,2}_{\Lip}([0,T]\times\R)$ the set of $C^{1,2}$ functions on $[0,T] \times \R$ that are bounded together with their derivatives. For all $l > 0$, the Hölder spaces $H^l(\R)$ and $H^{l/2,l}([0,T]\times\R)$ are defined as in~\cite[p.~7]{lady}.

\begin{lem}\label{lem:regul}
  Assume that the uniform ellipticity condition (D3) and the regularity condition (R2) hold, that $m$ has a finite first order moment and that $H*m$ is in the Hölder space $H^l(\R)$, with $l=3+\beta$. Then for all finite $T>0$, the probabilistic solution $F$ to~\eqref{eq:myPDE} is in $C^{1,2}_{\Lip}([0,T]\times\R)$. In particular, it is a classical solution to~\eqref{eq:myPDE}.
\end{lem}
\begin{proof}
  Fix a finite $T>0$. Then owing to the assumptions (D3), (R2) and on the regularity of $H*m$, the classical result of Lady\v{z}enskaja, Solonnikov and Ural'ceva~\cite[Theorem~8.1, p.~495]{lady} ensures that the Cauchy problem in divergence form
  \begin{equation}\label{eq:divCauchy}
    \left\{\begin{aligned}
      & \partial_t \tilde{F}_t(x) = \partial_x\left(\frac{1}{2}a(\tilde{F}_t(x))\partial_x \tilde{F}_t(x) - B(\tilde{F}_t(x))\right),\\
      & \tilde{F}_0(x) = H*m(x),
    \end{aligned}\right.
  \end{equation}
  admits a classical bounded solution $\tilde{F}$, which belongs to the Hölder space $H^{l/2,l}([0,T]\times\R)$, with $l=3+\beta$. Certainly, $\tilde{F}$ satisfies~\eqref{eq:weak} for all $g \in \Cc([0,T)\times\R)$. Let us now prove that $\tilde{F} \in \F(T)$. On the one hand, by the maximum principle~\cite[Theorem~2.5, p.~18]{lady}, for all $t \in [0,T]$, $||\tilde{F}_t||_{\infty} \leq 1$. On the other hand, the space derivative $\tilde{p} := \partial_x \tilde{F}$ is $C^{1,2}$ on $[0,T]\times\R$ and satisfies the linear parabolic equation
  \begin{equation*}
    \partial_t \tilde{p}_t(x) = \tilde{a}(t,x) \partial_x^2 \tilde{p}_t(x) + \tilde{b}(t,x) \partial_x \tilde{p}_t(x) + \tilde{c}(t,x) \tilde{p}_t(x),
  \end{equation*}
  where
  \begin{equation*}
    \begin{aligned}
      & \tilde{a}(t,x) := \frac{1}{2}a(\tilde{F}_t(x)),\\
      & \tilde{b}(t,x) := \frac{3}{2}a'(\tilde{F}_t(x))\partial_x \tilde{F}_t(x) - b(\tilde{F}_t(x)),\\
      & \tilde{c}(t,x) := \frac{1}{2}a''(\tilde{F}_t(x))(\partial_x \tilde{F}_t(x))^2 - b'(\tilde{F}_t(x))\partial_x \tilde{F}_t(x).
    \end{aligned}
  \end{equation*}
  The coefficients $\tilde{a}$, $\tilde{b}$ and $\tilde{c}$ are continuous and bounded in $[0,T] \times \R$ and, due to the condition (D3), the operator is parabolic. By the maximum principle~\cite[Theorem~9, p.~43]{friedman}, and since $\tilde{p}_0 \geq 0$, then $\tilde{p}_t(x) \geq 0$ for all $(t,x) \in [0,T]\times\R$.
  
  As a consequence, for all $t \in [0,T]$, $\tilde{p}_t$ is the density of a nonnegative bounded measure on $\R$, with total mass lower than $1$. Let us now prove that the mapping $t \mapsto \tilde{p}_t(x)\dd x$ is continuous for the topology of weak convergence. Since $\tilde{p}$ is continuous on $[0,T]\times\R$, the mapping $t \mapsto \tilde{p}_t(x)\dd x$ is continuous for the topology of vague convergence. Besides, as, for all $s,t \in [0,T]$, $\sup_{x \in \R} |F_t(x)-F_s(x)| \leq ||\partial_t F||_{\infty}|t-s|$, the total mass $t \mapsto \int_{\R} \tilde{p}_t(x)\dd x$ is continuous.
  
  Hence, the continuous mapping $t \mapsto \tilde{p}_t(x)\dd x$ is a measure-valued solution to the linear Fokker-Planck equation
  \begin{equation*}
    \partial_t \mu_t = \frac{1}{2}\partial_x^2\big(a(\tilde{F}_t(x))\mu_t\big) - \partial_x\big(b(\tilde{F}_t(x))\mu_t\big), 
  \end{equation*}
  the coefficients of which are measurable and bounded functions on $[0,T]\times\R$. Therefore by Figalli~\cite[Theorem~2.6]{figalli}, and since $\tilde{p}_0$ is a probability density, then for all $t \in [0,T]$, $\tilde{p}_t$ is the density of the distribution of $\tilde{X}_t$, where $(\tilde{X}_t)_{t \in [0,T]}$ is a weak solution to the stochastic differential equation
  \begin{equation*}
    \tilde{X}_t = \tilde{X}_0 + \int_0^t b(\tilde{F}_s(\tilde{X}_s))\dd s + \int_0^t \sigma(\tilde{F}_s(\tilde{X}_s))\dd \tilde{W}_s,
  \end{equation*}
  where $\tilde{X}_0$ has distribution $m$ and is independent of the Brownian motion $\tilde{W}$. Now one easily deduces from the assumption that $m$ has a finite first order moment and from the boundedness of $\sigma$ and $b$ that $\tilde{F} \in \F(T)$. Therefore, by the first part of Proposition~\ref{prop:unicite-edp}, $\tilde{F}$ is the restriction to $[0,T]$ of the probabilistic solution $F$ to~\eqref{eq:myPDE} given by Proposition~\ref{prop:sp}. Hence, $F \in C^{1,2}_{\Lip}([0,T]\times\R)$ and the fact that $F$ is a classical solution to~\eqref{eq:myPDE} now follows from the fact that $T$ is arbitrarily large.
\end{proof}

\begin{rk}
  The regularity assumption on the initial condition $H*m$ is far from being necessary for the probabilistic solution $F$ to have the classical regularity. For instance, it is known for the case of the viscous conservation law that $F$ has the classical regularity even for a discontinuous initial condition $H*m$ (see~\cite[Corollary~1.2]{jm}).
\end{rk}

\subsection{The nonlinear martingale problem}\label{ss:nlmp} The propagation of chaos result of Proposition~\ref{prop:sp} only deals with the flow of time marginals of the empirical distribution $\mu^n$. A natural further question is the convergence in $\P(C([0,+\infty),\R))$ towards the solution to a proper {\em nonlinear martingale problem}. 

Recall that the distribution of the random variable $\mu^n$ in $\P(C([0,+\infty),\R))$ is denoted by $\pi'^n$, and by Lemma~\ref{lem:tightness}, it is tight. Let $X$ refer to the canonical process on the probability space $C([0,+\infty),\R)$, namely $X_t(\omega) := \omega_t$ for all $\omega \in C([0,+\infty),\R)$. We shall also denote by $C^2_{\mathrm{b}}(\R)$ the space of $C^2$ functions $\phi: \R \to \R$ such that $\phi$, $\phi'$ and $\phi''$ are bounded.

\begin{defi}\label{defi:nlmp}
  A probability distribution $P \in \P(C([0,+\infty),\R))$ is called a solution to the nonlinear martingale problem if:
  \begin{itemize}
    \item $P_0 = m$;
    \item for all $\phi \in C^2_{\mathrm{b}}(\R)$, the process $M^{\phi}$ defined by $M_t^{\phi} := \phi(X_t) - \phi(X_0) - \int_0^t L(P_s)\phi(X_s)\dd s$ is a $P$-martingale, where, for all $\mu \in \P(\R)$, $L(\mu) \phi(x) := b(H*\mu(x))\phi'(x) + (1/2)a(H*\mu(x))\phi''(x)$;
    \item $\dd t$-almost everywhere, $P_t$ does not weight points.
  \end{itemize}
\end{defi}

Following~\cite[Lemma~1.2]{jourdain:porous}, if $P$ is a solution to the nonlinear martingale problem, then the function $(t,x) \mapsto H*P_t(x)$ is a weak solution to the Cauchy problem~\eqref{eq:myPDE}. Besides, since the coefficients $a$ and $b$ are bounded, it is easily seen that if $m$ has a finite first order moment, then this solution belongs to $\F$, therefore it coincides with the probabilistic solution $F$ given by Proposition~\ref{prop:sp}. 

Owing to Lévy's characterization of the Brownian motion, a probability distribution solving the nonlinear martingale problem is the distribution of a weak solution to the nonlinear stochastic differential equation~\eqref{eq:nldp}. Reciprocally, the distribution $P$ of a weak solution to~\eqref{eq:nldp} is a solution to the nonlinear martingale problem if and only if, $\dd t$-almost everywhere, $P_t$ does not weight points. When there exists a unique solution to the nonlinear martingale problem, we will refer to the associated weak solution $X$ of~\eqref{eq:nldp} as the {\em nonlinear diffusion process}.

\sk
Let us first investigate the existence of a solution to the nonlinear martingale problem.

\begin{lem}\label{lem:existnlmp}
  Under the nondegeneracy condition (D2) and the assumption that $m$ has a finite first order moment, the limit of any converging subsequence of $(\pi'^n)_{n \geq 1}$ concentrates on the set of solutions to the nonlinear martingale problem.
\end{lem}
\begin{proof}
  By Lemma~\ref{lem:tightness}, the sequence $(\pi'^n)_{n \geq 1}$ is tight. Let $\pi'^{\infty}$ denote the limit of a converging subsequence, that we still index by $n$ for convenience. Let $Q$ refer to the canonical variable in the probability space $\P(C([0,+\infty),\R))$. Since the variables $X^i_0$ are i.i.d. with marginal distribution $m$, then $\pi'^{\infty}$-a.s., $Q_0 = m$. Let us now prove that $\pi'^{\infty}$-a.s., $\dd t$-almost everywhere, $Q_t$ does not weight points. By Proposition~\ref{prop:sp}, $\pi'^{\infty}$-a.s., for all $t \geq 0$ one has $H*Q_t = F_t$ where $F$ is the probabilistic solution to the Cauchy problem~\eqref{eq:myPDE}. Therefore it is enough to prove that, $\dd t$-almost everywhere, the function $F_t$ is continuous on $\R$. In this purpose, we first remark that the mapping $t \mapsto P(t)$ solves a linear Fokker-Planck equation. Indeed, since $A$ and $B$ are $C^1$ on $[0,1]$, the functions $(t,x) \mapsto A(F_t(x))$, $(t,x) \mapsto B(F_t(x))$ are of finite variation and the associated Stieltjes measures write $\dd (A(F_t(x))) = \bar{a}(t,x)P(t)(\dd x)$, $\dd (B(F_t(x))) = \bar{b}(t,x)P(t)(\dd x)$, where
  \begin{equation*}
    \bar{a}(t,x) := \left\{\begin{aligned}
      & a(F_t(x)) & & \text{if $F_t$ is continuous in $x$,}\\
      & \frac{A(F_t(x)) - A(F_t(x^-))}{F_t(x)-F_t(x^-)} & & \text{otherwise,}
    \end{aligned}\right.
  \end{equation*}  
  and $\bar{b}(t,x)$ is similarly defined. Remark that the functions $\bar{a}$ and $\bar{b}$ are bounded, and $||\bar{a}||_{\infty} \leq ||a||_{\infty}$, $||\bar{b}||_{\infty} \leq ||b||_{\infty}$. As a consequence, the continuous mapping $t \mapsto P(t)$ is a measure-valued solution on $[0,+\infty)$ to the Fokker-Planck equation with measurable and bounded coefficients
  \begin{equation*}
    \partial_t P(t) = \frac{1}{2} \partial_x^2 (\bar{a}(t,x)P(t)) - \partial_x (\bar{b}(t,x)P(t)).
  \end{equation*}

  Fix a finite $T>0$. Then by Figalli~\cite[Theorem~2.6]{figalli}, there exists a probability distribution $\bar{P} \in \P(C([0,T],\R))$ such that, for all $t \in [0,T]$, $\bar{P}_t = P(t)$ and $\bar{P}$ is the distribution of a weak solution $(\bar{X}_t)_{t \in [0,T]}$ on $[0,T]$ to the stochastic differential equation 
  \begin{equation*}
    \bar{X}_t = \bar{X}_0 + \int_0^t \bar{b}(s,\bar{X}_s) \dd s + \int_0^t \bar{\sigma}(s,\bar{X}_s) \dd \bar{W}_s,
  \end{equation*}
  where $\bar{\sigma}(t,x) := \bar{a}(t,x)^{1/2}$. The process $(\bar{X}_t)_{t \in [0,T]}$ satisfies the condition of Bogachev, Krylov and Röckner~\cite[Remark~2.2.3, p.~63]{bkr}. Hence the positive measure $\bar{\sigma}(t,x) \bar{P}_t(\dd x) \dd t$ admits a density $\rho(t,x) \in L^2_{\mathrm{loc}}([0,T]\times\R)$. Now for all $t \in [0,T]$, $F_t$ is the cumulative distribution function of $\bar{X}_t$ so that $F_t(x) > 0$, $\bar{P}_t$-almost everywhere. By (D2) and the definition of $\bar{\sigma}$, we deduce that $\bar{P}_t$-almost everywhere, $\bar{\sigma}(t,x) > 0$, therefore 
  \begin{equation*}
    \bar{P}_t(\dd x)\dd t = \ind{\bar{\sigma}(t,x) > 0}\frac{\rho(t,x)}{\bar{\sigma}(t,x)}\dd x\dd t
  \end{equation*}
  and consequently, $\bar{P}_t(\dd x)$ admits a density $\dd t$-almost everywhere in $[0,T]$. We conclude by taking $T$ arbitrarily large.
  
  We finally prove that $\pi'^{\infty}$-a.s., for all $\phi \in C^2_{\mathrm{b}}(\R)$, the process $M^{\phi}$ defined by $M_t^{\phi} := \phi(X_t) - \phi(X_0) - \int_0^t L(Q_s)\phi(X_s)\dd s$ is a $Q$-martingale. We will proceed as in the proof of~\cite[Lemma~1.6]{jourdain:porous}. Let $\phi \in C^2_{\mathrm{b}}(\R)$, $k \geq 1$, $0 \leq s_1 \leq \cdots \leq s_k \leq s \leq t$ and $g : \R^k \to \R$ continuous and bounded. For all $Q \in \P(C([0,+\infty),\R))$, we define
  \begin{equation*}
    \G(Q) := \left\langle Q, g(X_{s_1}, \ldots, X_{s_k})\left(\phi(X_t) - \phi(X_s) - \int_s^t L(Q_r)\phi(X_r)\dd r \right) \right\rangle.
  \end{equation*}
  By Itô's formula, for all $n \geq 1$,
  \begin{equation*}
    \begin{split}
      \G(\mu^n) = \frac{1}{n} \sum_{i=1}^n g(X_{s_1}, \ldots, X_{s_k})\bigg(&\int_s^t \phi'(X_r^{i,n}) \left(c_n + \sigma(H*\mu^n_t(X_r^{i,n}))\right)\dd W_r^i\\
      & + \int_s^t \phi''(X_r^{i,n})\left(c_n\sigma(H*\mu^n_t(X_r^{i,n})) + \frac{c_n^2}{2}\right)\dd r \bigg),
    \end{split}
  \end{equation*}
  so that, since $(c_n)_{n \geq 1}$, $\sigma$, $g$, $\phi'$ and $\phi''$ are bounded, $\lim_{n \to +\infty} \Exp(\G(\mu^n)^2) = 0$. We now check that the functional $\G$ is continuous at all $P \in \P(C([0,+\infty),\R))$ such that, $\dd r$-almost everywhere, $P_r$ does not weight points. Let $(P^q)_{q \geq 1}$ be a sequence of probability distributions on $C([0,+\infty),\R)$ weakly converging to $P \in \P(C([0,+\infty),\R))$ such that, $\dd r$-almost everywhere, $P_r$ does not weight points. Then, for all $q \geq 1$,
  \begin{equation}\label{eq:GPq}
    \begin{aligned}
      \G(P^q) &= \left\langle P^q, g(X_{s_1}, \ldots, X_{s_k})\int_s^t (L(P^q_r)-L(P_r))\phi(X_r)\dd r\right\rangle\\
      & \quad + \left\langle P^q, g(X_{s_1}, \ldots, X_{s_k})\left(\phi(X_t)-\phi(X_s)-\int_s^t L(P_r)\phi(X_r)\dd r\right)\right\rangle.
    \end{aligned}
  \end{equation}
  On the one hand, as $g$ and the derivatives of $\phi$ are bounded, there exists $C > 0$ independent of $q$ such that, for all $q \geq 1$,
  \begin{equation*}
    \begin{aligned}
      & \left| \left\langle P^q, g(X_{s_1}, \ldots, X_{s_k})\int_s^t (L(P^q_r)-L(P_r))\phi(X_r)\dd r\right\rangle \right|\\
      & \qquad \leq C \int_s^t \left(\sup_{x \in \R} |b(H*P^q_r(x)) - b(H*P_r(x))| + \sup_{x \in \R} |a(H*P^q_r(x)) - a(H*P_r(x))|\right) \dd r.
    \end{aligned}
  \end{equation*}
  By Dini's theorem, $\dd r$-almost everywhere in $[s,t]$, $H*P^q_r$ converges uniformly to $H*P_r$ on $\R$ when $q \to +\infty$. As the functions $b$ and $a$ are bounded and uniformly continuous on $[0,1]$, by Lebesgue's theorem, the right-hand side above goes to $0$ when $q \to +\infty$.
  
  On the other hand, $\dd r$-almost everywhere in $[s,t]$, the function $x \mapsto L(P_r)\phi(x)$ is continuous on $\R$ and uniformly bounded in $r$, therefore by Lebesgue's theorem again, the function $(x_r)_{r \geq 0} \mapsto g(x_{s_1}, \ldots, x_{s_k})(\phi(x_t)-\phi(x_s)-\int_s^t L(P_r)\phi(x_r)\dd r)$ is continuous on $C([0,+\infty),\R)$. As a consequence, the second term in the right-hand side of~\eqref{eq:GPq} converges to $\G(P)$ and we conclude that $\lim_{q \to +\infty} \G(P^q) = \G(P)$.
  
  Since we have proved that $\pi'^{\infty}$-a.s., $\dd r$-almost everywhere, $Q_r$ does not weight points, then
  \begin{equation*}
    \lim_{n \to +\infty} \Exp^{\pi'^n}(\G(Q)^2) = \Exp^{\pi'^{\infty}}(\G(Q)^2),
  \end{equation*}
  which rewrites
  \begin{equation*}
     \Exp^{\pi'^{\infty}}(\G(Q)^2) = \lim_{n \to +\infty} \Exp(\G(\mu^n)^2) = 0.
  \end{equation*}
  As a consequence, taking $\phi$, $(s_1, \ldots, s_k, s, t)$, $g$ in countable subsets leads to the conclusion that $\pi'^{\infty}$-a.s., $Q$ solves the nonlinear martingale problem.
\end{proof}

We now address the uniqueness of solutions to the nonlinear martingale problem. The following criterion is due to Stroock and Varadhan~\cite{stroock}.

\begin{lem}\label{lem:sv}
  Under the assumptions of Proposition~\ref{prop:sp}, if the function $(t,x) \mapsto a(F_t(x))$ is uniformly positive on the compact sets of $[0,+\infty)\times\R$, then there is at most one solution to the nonlinear martingale problem.
\end{lem}
\begin{proof}
  Let $P$ and $Q$ denote two solutions to the nonlinear martingale problem. Then they both solve the following linear martingale problem in $R \in \P(C([0,+\infty),\R))$:
  \begin{itemize}
    \item $R_0 = m$;
    \item for all $\phi \in C^2_{\mathrm{b}}(\R)$, the process $\tilde{M}^{\phi}$ defined by
    \begin{equation*}
      \tilde{M}_t^{\phi} := \phi(X_t) - \phi(X_0) - \int_0^t \left\{\phi'(X_s)b(F_s(X_s)) + \frac{1}{2}\phi''(X_s)a(F_s(X_s))\right\}\dd s
    \end{equation*}
    is a $R$-martingale.
  \end{itemize}
  The functions $a(F_s(x))$ and $b(F_s(x))$ are measurable and bounded, and $a(F_s(x))$ is uniformly positive on the compact sets of $[0,+\infty)\times\R$. By~\cite[Exercise~7.3.3, p.~192]{stroock}, $P=Q$.
\end{proof}

\begin{lem}\label{lem:unicitenlmp}
  Assume that $m$ has a finite first order moment, and either the uniform ellipticity condition (D3) holds, or the nondegeneracy condition (D2) holds and $F_0(x)>0$ for all $x \in \R$. Then the function $(t,x) \mapsto a(F_t(x))$ is uniformly positive on the compact sets of $[0,+\infty)\times\R$.
\end{lem}
\begin{proof}
  If (D3) holds, the result is obvious. Now if $F_0(x)>0$, for all compact subset $K \in [0,+\infty) \times \R$, by the first part of Lemma~\ref{lem:FtF0} in Appendix~\ref{app:quant}, there exists $u_0>0$ such that for all $(t,x) \in K$, $F_t(x) \geq u_0$. If (D2) holds in addition, for all $(t,x) \in K$, $a(F_t(x)) \geq \inf_{u \geq u_0} a(u) > 0$. 
\end{proof}

We conclude this subsection by stating a propagation of chaos result for the empirical distribution $\mu^n$ of the particle system in $\P(C([0,+\infty),\R))$.

\begin{cor}\label{cor:nlmp}
  Under the assumptions of Lemma~\ref{lem:unicitenlmp}, there exists a unique solution $P$ to the nonlinear martingale problem, and it is the limit in probability, in $\P(C([0,+\infty),\R))$, of the sequence of empirical distributions $\mu^n$.
\end{cor}

\subsection{The reordered particle system}\label{ss:spr} For all $t \geq 0$, let $Y_t^n := (Y_t^{1,n}, \ldots, Y_t^{n,n})$ denote the increasing reordering of the vector $(X_t^{1,n}, \ldots, X_t^{n,n})$. Then the sample-paths of the process $Y^n$ are in $C([0,+\infty),D_n)$, where $D_n$ refers to the polyhedron $\{(y_1, \ldots, y_n) \in \R^n : y_1 \leq \cdots \leq y_n\}$. 

It is well known that $Y^n$ is a normally reflected Brownian motion on $\partial D_n$, with constant drift vector and constant diagonal diffusion matrix. More precisely, according to~\cite{jourdain:porous}, the process $(\beta^1, \ldots, \beta^n)$ defined by $\beta^i_t = \sum_{j=1}^n \int_0^t \ind{X_s^{j,n} = X_s^{i,n}}\dd W_s^j$ is a Brownian motion. By the Itô-Tanaka formula,
\begin{equation}\label{eq:spr}
  Y_t^{i,n} = Y_0^{i,n} + b(i/n) t + (c_n+\sigma(i/n))\beta^i_t + V_t^i,
\end{equation}
where $V$ is a $\R^n$-valued continuous process with finite variation $|V|$ which writes $V_t^i = \int_0^t (\gamma^i_s - \gamma^{i+1}_s)\dd|V|_s$ with $\dd|V|_t$-a.e., $\gamma^1_t = \gamma^{n+1}_t = 0$, $\gamma^i_t \geq 0$ and $\gamma^i_t(Y_t^{i,n}-Y_t^{i-1,n}) = 0$. We shall now refer to the process $Y^n$ as the {\em reordered particle system} and denote by $\tmu^n \in \P(C([0,+\infty),\R))$ its empirical distribution. 

\begin{lem}[Tanaka~\cite{tanaka}]\label{lem:tanaka}
  For a given random variable $Y_0^n \in D_n$ and an independent $\R^n$-valued Brownian motion $(\beta^1, \ldots, \beta^n)$, there exists a unique process $(Y^n,V) \in C([0,+\infty),D_n\times\R^n)$ satisfying all the above conditions.
\end{lem}

For $Q \in \P(C([0,+\infty),\R))$ and $t_1, \ldots, t_k \geq 0$, let us denote by $Q_{t_1, \ldots, t_k} \in \P(\R^k)$ the finite-dimensional marginal distribution of $Q$. Let us define $\A$ as the set of probability distributions $Q \in \P(C([0,+\infty),\R))$ such that, for all $0 \leq t_1 < \cdots < t_k$, $Q_{t_1, \ldots, t_k}$ is the distribution of $(H*Q_{t_1})^{-1}(U), \ldots, (H*Q_{t_k})^{-1}(U))$ where $U$ is a uniform random variable on $[0,1]$. Remark that any $Q \in \A$ is exactly determined by the flow of its one-dimensional marginals $t \mapsto Q_t$.  

\begin{prop}\label{prop:spr}
  Under the assumptions of Proposition~\ref{prop:sp}, the empirical distribution $\tmu^n$ of the reordered particle system converges in probability, in $\P(C([0,+\infty),\R))$, to the unique $\tP \in \A$ such that for all $t \geq 0$, $\tP_t=P(t)$, where the mapping $t \mapsto P(t)$ is given by Proposition~\ref{prop:sp}. In particular, for all $t \geq 0$, $H*\tP_t = F_t$ where $F$ is the probabilistic solution to the Cauchy problem~\eqref{eq:myPDE}.
\end{prop}
\begin{proof}
  Let $\tpi^n$ refer to the distribution of $\tmu^n$ in $\P(C([0,+\infty),\R))$. According to Sznitman~\cite{sznitman}, the tightness of $(\tpi^n)_{n \geq 1}$ is equivalent to the tightness of the sequence of the distributions of the variables $Y^{\theta_n,n} \in C([0,+\infty),\R)$ where $\theta_n$ is a uniform random variable in the set $\{1, \ldots, n\}$, independent of $Y^n$. For all $n \geq 1$, $Y^{\theta_n,n}_0$ has distribution $m$. Besides, for $s \leq t$ and $p \geq 1$, 
  \begin{equation*}
    \Exp\left(|Y^{\theta_n,n}_s-Y^{\theta_n,n}_t|^p\right) = \frac{1}{n}\sum_{i=1}^n \Exp\left(|Y^{i,n}_s-Y^{i,n}_t|^p\right) = \Exp\left(W_p^p(\tmu^n_s,\tmu^n_t)\right).
  \end{equation*}
  Now, by~\eqref{eq:Wmunmun} and exchangeability of $(X^{1,n}_t)_{t \geq 0}, \ldots, (X^{n,n}_t)_{t \geq 0}$,
  \begin{equation*}
    \Exp\left(W_p^p(\tmu^n_s,\tmu^n_t)\right) \leq \frac{1}{n} \sum_{i=1}^n \Exp\left(|X_s^{i,n} - X_t^{i,n}|^p\right) = \Exp\left(|X_s^{1,n} - X_t^{1,n}|^p\right),
  \end{equation*}
  so that, by the proof of Lemma~\ref{lem:tightness} and the Kolmogorov criterion, the sequence $(\tpi^n)_{n \geq 1}$ is tight. 
  
  It is clear from the definition of the reordered particle system that for all $t \geq 0$, $\mu^n_t=\tmu^n_t$, therefore it is already known from Proposition~\ref{prop:sp} that $\tmu^n_t$ converges in distribution to $P(t)$. Consequently, the proof of Proposition~\ref{prop:spr} requires nothing but a uniqueness result for the support of any limit point of $(\tpi^n)_{n \geq 1}$.

  The latter is a consequence of the following remark. Certainly, $\tmu^n \in \A$, and the set $\A$ is closed in $\P(C([0,+\infty),\R))$ by~\cite[Lemma~3.5]{jourdain:characteristics}. Hence, any limit point of $(\tpi^n)_{n \geq 1}$ is concentrated on the unique probability distribution $\tP \in \A$ such that for all $t \geq 0$, $\tP_t = P(t)$.
\end{proof}

\subsection{Propagation of chaos in Wasserstein distance}\label{ss:cvWp} The original particle system defined by~\eqref{eq:sp} is exchangeable, therefore the propagation of chaos result stated in Proposition~\ref{prop:sp} implies that the distribution $P_t^{1,n}$ of $X_t^{1,n}$ converges weakly to $P(t)$ in $\P(\R)$. This convergence result can be strengthened in Wasserstein distance.

\begin{cor}\label{cor:cvWp}
  Under the nondegeneracy condition (D1) and the assumption that $m$ has a finite moment of order $p \geq 1$, then $P^{1,n}_t$ and $P(t)$ have a finite moment of order $p$, and 
  \begin{equation*}
    \lim_{n \to +\infty} W_p(P^{1,n}_t,P(t)) = 0, \qquad \lim_{n \to +\infty} \Exp[W_p^p(\mu_t^n, P(t))] \to 0.
  \end{equation*}
\end{cor}
\begin{proof}
  Let $t \geq 0$. As just seen before the corollary, $P_t^{1,n}$ converges weakly to $P(t)$. To prove that this convergence holds in the Wasserstein distance of order $p$, it is sufficient to prove that the sequence $(|X_t^{1,n}|^p)_{n \geq 1}$ is uniformly integrable (see Villani~\cite[Theorem~6.9, p.~108]{villani}). For all $q \geq 1$,
  \begin{equation}\label{eq:EXtX0}
    \Exp\left(|X_t^{1,n}-X_0^1|^q\right) \leq 2^{q-1}\left((t||b||_{\infty})^q + \Exp\left(\left|\int_0^t \left(c_n + \sigma(H*\mu_s^n(X_s^{1,n}))\right) \dd W_s^1\right|^q\right)\right) \leq C,
  \end{equation}
  where $C$ does not depend on $n$. Thus, the sequence $(|X_t^{1,n}-X_0^1|^p)_{n \geq 1}$ is uniformly integrable, and since $|X_t^{1,n}|^p \leq 2^{p-1}(|X_t^{1,n}-X_0^1|^p + |X_0^1|^p)$ then the sequence $(|X_t^{1,n}|^p)_{n \geq 1}$ is uniformly integrable. Therefore $P_t^{1,n}$ and $P(t)$ have a finite moment of order $p$ and $W_p(P_t^{1,n},P(t)) \to 0$.
  
  Let $M \geq 0$. Then, by~\eqref{eq:Wmunmu},
  \begin{equation*}
    \begin{split}
      & W_p^p(\mu_t^n, P(t)) = \sum_{i=1}^n \int_{(i-1)/n}^{i/n} (|Y_t^{i,n}-F_t^{-1}(u)|^p-M)^+\dd u + \sum_{i=1}^n \int_{(i-1)/n}^{i/n} (|Y_t^{i,n}-F_t^{-1}(u)|^p \wedge M)\dd u\\
      & \leq \sum_{i=1}^n \int_{(i-1)/n}^{i/n} |Y_t^{i,n}-F_t^{-1}(u)|^p \ind{|Y_t^{i,n}-F_t^{-1}(u)|^p \geq M}\dd u + \sum_{i=1}^n \int_{(i-1)/n}^{i/n} (|Y_t^{i,n}-F_t^{-1}(u)|^p \wedge M)\dd u.
    \end{split}
  \end{equation*}
  
  On the one hand,
  \begin{equation*}
    \sum_{i=1}^n \int_{(i-1)/n}^{i/n} (|Y_t^{i,n}-F_t^{-1}(u)|^p \wedge M)\dd u = \int_0^1 (|(H*\mu_t^n)^{-1}(u)-F_t^{-1}(u)|^p \wedge M)\dd u,
  \end{equation*}
  and the function $\mu \in \P(\R) \mapsto \int_0^1 (|(H*\mu)^{-1}(u)-F_t^{-1}(u)|^p \wedge M)\dd u$ is continuous and bounded. Therefore, by Proposition~\ref{prop:sp},
  \begin{equation*}
    \lim_{n \to +\infty}\Exp\left(\sum_{i=1}^n \int_{(i-1)/n}^{i/n} (|Y_t^{i,n}-F_t^{-1}(u)|^p \wedge M)\dd u\right) = 0.
  \end{equation*}
  
  On the other hand, remarking that for all $x,y \in \R$,
  \begin{equation}\label{eq:xyp}
    \begin{split}
      |x-y|^p\ind{|x-y|^p \geq M} & \leq |x-y|^p\ind{|x| \geq |y|\vee M^{1/p}/2} + |x-y|^p\ind{|y| \geq |x|\vee M^{1/p}/2}\\
      & \leq 2^p|x|^p\ind{|x|^p \geq M/2^p} + 2^p|y|^p\ind{|y|^p \geq M/2^p},
    \end{split}
  \end{equation}
  we write
  \begin{equation*}
    \begin{split}
      & \sum_{i=1}^n \int_{(i-1)/n}^{i/n} |Y_t^{i,n}-F_t^{-1}(u)|^p\ind{|Y_t^{i,n}-F_t^{-1}(u)|^p \geq M}\dd u\\
      & \qquad \leq \frac{2^p}{n} \sum_{i=1}^n |Y_t^{i,n}|^p\ind{|Y_t^{i,n}|^p \geq M/2^p} + 2^p \int_0^1 |F_t^{-1}(u)|^p\ind{|F_t^{-1}(u)|^p \geq M/2^p}\dd u\\
      & \qquad = \frac{2^p}{n} \sum_{i=1}^n |X_t^{i,n}|^p\ind{|X_t^{i,n}|^p \geq M/2^p} + 2^p \int_0^1 |F_t^{-1}(u)|^p\ind{|F_t^{-1}(u)|^p \geq M/2^p}\dd u.
    \end{split}
  \end{equation*}
  We deduce from the exchangeability of the variables $X_t^{1,n}, \ldots, X_t^{n,n}$, the uniform integrability of $(|X_t^{1,n}|^p)_{n \geq 1}$ and the finiteness of $\int_0^1 |F_t^{-1}(u)|^p \dd u = \int_{\R} |x|^pP(t)(\dd x)$ that
  \begin{equation*}
    \lim_{M \to +\infty} \sup_{n \geq 1} \Exp\left(\sum_{i=1}^n \int_{(i-1)/n}^{i/n} |Y_t^{i,n}-F_t^{-1}(u)|^p\ind{|Y_t^{i,n}-F_t^{-1}(u)|^p \geq M}\dd u\right) = 0,
  \end{equation*}
  so that $\Exp[W_p^p(\mu_t^n, P(t))] \to 0$.
\end{proof}

\section{Contraction of the Wasserstein distance between two solutions}\label{s:contr}

Let $F_0$ and $G_0$ be the cumulative functions of two probability distributions with a finite first order moment. Under the condition (D1), by Proposition~\ref{prop:sp} there exist a unique probabilistic solution $F$ to the Cauchy problem~\eqref{eq:myPDE} with initial condition $F_0$, and a unique probabilistic solution $G$ to the Cauchy problem~\eqref{eq:myPDE} with initial condition $G_0$. This section addresses the behaviour of the flow $t \mapsto W_p(F_t,G_t)$. In Proposition~\ref{prop:contr} we prove that it is nonincreasing if $W_p(F_0,G_0) < +\infty$, using only the contractivity of the reordered particle system. Then, assuming the classical regularity of $F$ and $G$, we provide an explicit expression of the time derivative of the flow $t \mapsto W_p^p(F_t,G_t)$.

We point out the fact that we will sometimes call {\em expectation} or {\em moment} of a cumulative distribution function the expectation or the moment of the derivated probability distribution.

\subsection{Monotonicity of the flow} We first deduce from a natural coupling between two versions of the reordered particle system that the flow $t \mapsto W_p(F_t,G_t)$ is nonincreasing.

\begin{prop}\label{prop:contr}
  Assume that the nondegeneracy condition (D1) holds and that $F_0$ and $G_0$ have a finite first order moment. Then, for all $p \geq 1$,
  \begin{itemize}
    \item if $W_p(F_0,G_0) < +\infty$, then the flow $t \mapsto W_p(F_t,G_t)$ in nonincreasing;
    \item if $W_p(F_0,G_0) = +\infty$, then for all $t \geq 0$, $W_p(F_t,G_t) = +\infty$.
  \end{itemize}
\end{prop}
\begin{proof}
  We deduce the monotonicity property from the contractive behaviour of the reordered particle system. Let $(\beta^1, \ldots, \beta^n)$ be a $\R^n$-valued Brownian motion and let $U^1, \ldots, U^n$ be independent uniform variables on $[0,1]$. Let us denote by $(U^{(1)}, \ldots, U^{(n)})$ the increasing reordering of $(U^1, \ldots, U^n)$. For all $1 \leq i \leq n$, let $Y_0^{F,i} := F_0^{-1}(U^{(i)})$ and $Y_0^{G,i} := G_0^{-1}(U^{(i)})$. By Lemma~\ref{lem:tanaka}, there exists a unique strong solution $(Y^F,V^F) \in C([0,+\infty),D_n\times\R^n)$ to the reflected stochastic differential equation 
  \begin{equation*}
    Y_t^{F,i,n} = Y_0^{F,i} + b(i/n) t + (c_n + \sigma(i/n))\beta^i_t + V^{F,i}_t,
  \end{equation*}
  and similarly, we denote by $(Y^G,V^G)$ the unique strong solution in $C([0,+\infty),D_n\times\R^n)$ to the reflected stochastic differential equation
  \begin{equation*}
    Y_t^{G,i,n} = Y_0^{G,i} + b(i/n) t + (c_n + \sigma(i/n))\beta^i_t + V^{G,i}_t.
  \end{equation*}
  By the beginning of Subsection~\ref{ss:spr} and the Yamada-Watanabe theorem, the process $Y^{F,n}$ (resp. $Y^{G,n}$) has the same distribution as the increasing reordering of the particle system $X^{F,n}$ (resp. $X^{G,n}$) solution to~\eqref{eq:sp} with initial conditions i.i.d. according to the cumulative distribution function $F_0$ (resp. $G_0$). In particular, the propagation of chaos result of Proposition~\ref{prop:spr} applies to the empirical distributions $\tmu^{F,n} := (1/n)\sum_{i=1}^n \delta_{Y^{F,i,n}}$ and $\tmu^{G,n} := (1/n)\sum_{i=1}^n \delta_{Y^{G,i,n}}$.
  
  Now, for all $t \geq 0$, \eqref{eq:Wmunmun} yields $W_p^p(\tmu^{F,n}_t, \tmu^{G,n}_t) = (1/n)\sum_{i=1}^n |Y_t^{F,i,n} - Y_t^{G,i,n}|^p$. The following inequality, the proof of which is postponed below, is crucial: 
  \begin{equation}\label{eq:decspr}
    \forall 0 \leq s \leq t, \qquad W_p^p(\tmu^{F,n}_t, \tmu^{G,n}_t) \leq W_p^p(\tmu^{F,n}_s, \tmu^{G,n}_s).
  \end{equation}
  
  \sk
  {\bf Case $W_p(F_0,G_0) < +\infty$.} If both $F_0$ and $G_0$ have a finite moment of order $p$, then owing to Corollary~\ref{cor:cvWp}, one can extract a subsequence along which $W_p^p(\tmu^{F,n}_t, \tmu^{G,n}_t)$ goes to $W_p^p(F_t,G_t)$ and $W_p^p(\tmu^{F,n}_s, \tmu^{G,n}_s)$ goes to $W_p^p(F_s,G_s)$ almost surely, then conclude by using~\eqref{eq:decspr}. 
  
  Assuming only $W_p(F_0,G_0) < +\infty$, we shall now proceed as in the proof of Corollary~\ref{cor:cvWp} to show that for all $t \geq 0$, 
  \begin{equation*}
    \lim_{n \to +\infty} \Exp(W_p^p(\tmu^{F,n}_t, \tmu^{G,n}_t)) = W_p^p(F_t,G_t),
  \end{equation*}
  which results in the claimed assertion thanks to~\eqref{eq:decspr}.
  
  For all $M \geq 0$, by~\eqref{eq:Wmunmun} we write
  \begin{equation*}
    \begin{split}
      W_p^p(\tmu_t^{F,n},\tmu_t^{G,n}) & = \frac{1}{n}\sum_{i=1}^n |Y_t^{F,i,n}-Y_t^{G,i,n}|^p\\
      & = \frac{1}{n}\sum_{i=1}^n (|Y_t^{F,i,n}-Y_t^{G,i,n}|^p \wedge M) + \left(|Y_t^{F,i,n}-Y_t^{G,i,n}|^p-M\right)^+.
    \end{split}
  \end{equation*}
  
  By Proposition~\ref{prop:sp}, $\tmu_t^{F,n}$ converges in probability to the probability distribution $\dd F_t$ with cumulative distribution function $F_t$, and similarly $\tmu_t^{G,n}$ converges in probability to $\dd G_t$. Therefore, the couple $(\tmu_t^{F,n},\tmu_t^{G,n})$ converges in probability to $(\dd F_t, \dd G_t)$ and
  \begin{equation*}
    \lim_{n \to +\infty} \Exp\left(\frac{1}{n}\sum_{i=1}^n (|Y_t^{F,i,n}-Y_t^{G,i,n}|^p \wedge M)\right) = \int_0^1 (|F_t^{-1}(u)-G_t^{-1}(u)|^p \wedge M) \dd u.
  \end{equation*}
  By the monotone convergence theorem and~\eqref{eq:WinvCDF},
  \begin{equation*}
    \lim_{M \to +\infty} \lim_{n \to +\infty} \Exp\left(\frac{1}{n}\sum_{i=1}^n (|Y_t^{F,i,n}-Y_t^{G,i,n}|^p \wedge M)\right) = W_p^p(F_t,G_t) \in [0,+\infty].
  \end{equation*}
  
  It now remains to check that
  \begin{equation}\label{eq:contr:1}
    \lim_{n \to +\infty} \Exp\left(\frac{1}{n}\sum_{i=1}^n (|Y_t^{F,i,n}-Y_t^{G,i,n}|^p-M)^+\right) = 0.
  \end{equation}
  Using~\eqref{eq:xyp} twice results in
  \begin{equation*}
    \begin{split}
      \frac{1}{n}\sum_{i=1}^n \left(|Y_t^{F,i,n}-Y_t^{G,i,n}|^p-M\right)^+ & \leq \frac{1}{n}\sum_{i=1}^n |Y_t^{F,i,n}-Y_t^{G,i,n}|^p\ind{|Y_t^{F,i,n}-Y_t^{G,i,n}|^p \geq M}\\
      & \leq \frac{4^p}{n}\sum_{i=1}^n |Y_t^{F,i,n}-Y_0^{F,i}|^p\ind{|Y_t^{F,i,n}-Y_0^{F,i}|^p \geq M/4^p}\\
      & \quad + \frac{4^p}{n}\sum_{i=1}^n |Y_t^{G,i,n}-Y_0^{G,i}|^p\ind{|Y_t^{G,i,n}-Y_0^{G,i}|^p \geq M/4^p}\\
      & \quad + \frac{2^p}{n}\sum_{i=1}^n |Y_0^{F,i}-Y_0^{G,i}|^p\ind{|Y_0^{F,i}-Y_0^{G,i}|^p \geq M/2^p}.
    \end{split}
  \end{equation*}
  On the one hand, by the construction of $Y_0^{F,i}$ and $Y_0^{G,i}$,
  \begin{equation*}
    \begin{split}
      & \Exp\left(\frac{1}{n}\sum_{i=1}^n |Y_0^{F,i}-Y_0^{G,i}|^p\ind{|Y_0^{F,i}-Y_0^{G,i}|^p \geq M/2^p}\right)\\
      & \qquad = \Exp\left(\frac{1}{n}\sum_{i=1}^n |F_0^{-1}(U^{(i)})-G_0^{-1}(U^{(i)})|^p\ind{|F_0^{-1}(U^{(i)})-G_0^{-1}(U^{(i)})|^p \geq M/2^p}\right)\\
      & \qquad = \Exp\left(\frac{1}{n}\sum_{i=1}^n |F_0^{-1}(U^i)-G_0^{-1}(U^i)|^p\ind{|F_0^{-1}(U^i)-G_0^{-1}(U^i)|^p \geq M/2^p}\right)\\
      & \qquad = \int_0^1 |F_0^{-1}(u)-G_0^{-1}(u)|^p\ind{|F_0^{-1}(u)-G_0^{-1}(u)|^p \geq M/2^p} \dd u,
    \end{split}
  \end{equation*}
  and the right-hand side does not depend on $n$. Since $W_p(F_0,G_0) < +\infty$, it goes to $0$ when $M \to +\infty$. On the other hand,
  \begin{equation*}
    \begin{split}
      \Exp\left(\frac{1}{n}\sum_{i=1}^n |Y_t^{F,i,n}-Y_0^{F,i}|^p\ind{|Y_t^{F,i,n}-Y_0^{F,i}|^p \geq M/4^p}\right) & \leq \frac{4^p}{M}\Exp\left(\frac{1}{n}\sum_{i=1}^n |Y_t^{F,i,n}-Y_0^{F,i}|^{p+1}\right)\\
      & \leq \frac{4^p}{M}\Exp\left(\frac{1}{n}\sum_{i=1}^n |X_t^{F,i,n}-X_0^{F,i}|^{p+1}\right),   
    \end{split}
  \end{equation*}
  where we have used the inequality~\eqref{eq:Wmunmun} in the last line. By~\eqref{eq:EXtX0}, there exists $C > 0$ independent of $n$ such that $\Exp(|X_t^{F,i,n}-X_0^{F,i}|^{p+1}) \leq C$. Then
  \begin{equation*}
    \lim_{M \to +\infty} \limsup_{n \to +\infty} \Exp\left(\frac{1}{n}\sum_{i=1}^n |Y_t^{F,i,n}-Y_0^{F,i}|^p\ind{|Y_t^{F,i,n}-Y_0^{F,i}|^p \geq M/4^p}\right) = 0;
  \end{equation*}
  and likewise,
  \begin{equation*}
    \lim_{M \to +\infty} \limsup_{n \to +\infty} \Exp\left(\frac{1}{n}\sum_{i=1}^n |Y_t^{G,i,n}-Y_0^{G,i}|^p\ind{|Y_t^{G,i,n}-Y_0^{G,i}|^p \geq M/4^p}\right) = 0,
  \end{equation*}
  which completes the proof of~\eqref{eq:contr:1}.

  \sk
  {\bf Case $W_p(F_0,G_0) = +\infty$.} By the triangle inequality,
  \begin{equation*}
    W_p(F_0,G_0) \leq W_p(F_0,F_t) + W_p(G_0,G_t) + W_p(F_t,G_t).
  \end{equation*}
  According to~Proposition~\ref{prop:spr}, there exists a subsequence (that we still index by $n$ for convenience) along which $\tmu_0^{F,n}$ converges to the distribution with cumulative function $F_0$ almost surely in $\P(\R)$, and $\tmu_t^{F,n}$ converges to the distribution with cumulative function $F_t$ almost surely in $\P(\R)$. Recalling that the Wasserstein distance is lower semicontinuous on $\P(\R)$ (see~\cite[Remark~6.12]{villani}), we get $W_p(F_0,F_t) \leq \liminf_{n \to +\infty} W_p(\tmu_0^{F,n},\tmu_t^{F,n})$ so that by Fatou's lemma, \eqref{eq:Wmunmun} and~\eqref{eq:EXtX0},
  \begin{equation*}
    W_p^p(F_0,F_t) \leq \liminf_{n \to +\infty} \Exp\left(\frac{1}{n} \sum_{i=1}^n |Y_t^{F,i,n}-Y_0^{F,i}|^p\right) < +\infty,
  \end{equation*}
  and similarly $W_p(G_0,G_t) < +\infty$. As a consequence, if $W_p(F_0,G_0) = +\infty$ then $W_p(F_t,G_t) = +\infty$.

  \sk
  {\bf Proof of~\eqref{eq:decspr}.} Recall that
  \begin{equation*}
    \begin{split}
      & Y_t^{F,i,n} = Y_0^{F,i} + b(i/n) t + (c_n + \sigma(i/n))\beta^i_t + V^{F,i}_t,\\
      & Y_t^{G,i,n} = Y_0^{G,i} + b(i/n) t + (c_n + \sigma(i/n))\beta^i_t + V^{G,i}_t,
    \end{split}
  \end{equation*}
  with $\dd V_t^{F,i} = (\gamma_t^{F,i}-\gamma_t^{F,i+1})\dd |V^F|_t$, $\dd V_t^{G,i} = (\gamma_t^{G,i}-\gamma_t^{G,i+1})\dd |V^G|_t$. Thus, for all $1 \leq i \leq n$, the process $Y^{F,i,n} - Y^{G,i,n}$ is of finite variation, hence
  \begin{equation*}
    \begin{split}
      \dd \sum_{i=1}^n |Y^{F,i,n}_t - Y^{G,i,n}_t|^p & = p\sum_{i=1}^n |Y^{F,i,n}_t - Y^{G,i,n}_t|^{p-2}(Y^{F,i,n}_t - Y^{G,i,n}_t)\dd (Y^{F,i,n}_t - Y^{G,i,n}_t)\\
      &= p\sum_{i=1}^n |Y^{F,i,n}_t - Y^{G,i,n}_t|^{p-2}(Y^{F,i,n}_t - Y^{G,i,n}_t)(\gamma^{F,i}_t - \gamma^{F,i+1}_t)\dd|V^F|_t\\ 
      &\quad + p\sum_{i=1}^n |Y^{G,i,n}_t - Y^{F,i,n}_t|^{p-2}(Y^{G,i,n}_t - Y^{F,i,n}_t)(\gamma^{G,i}_t - \gamma^{G,i+1}_t)\dd|V^G|_t,
    \end{split}
  \end{equation*}
  where, for all $p \geq 1$, we take the convention that $|z|^{p-2}z=0$ when $z=0$. The two terms of the last member above are symmetric; we only deal with the first one, which we rewrite $S_t\dd|V^F|_t$. By the Abel transform, $S_t = p\sum_{i=2}^n \gamma^{F,i}_t u(Y_t^{F,i-1,n},Y_t^{G,i-1,n},Y_t^{F,i,n},Y_t^{G,i,n})$, where $u(x_1,y_1,x_2,y_2) := |x_2-y_2|^{p-2}(x_2-y_2) - |x_1-y_1|^{p-2}(x_1-y_1)$. Recall that $\dd|V^F|_t$-a.e., $\gamma^{F,i}_t \geq 0$ and $\gamma^{F,i}_t(Y^{F,i,n}_t - Y^{F,i-1,n}_t) = 0$, and remark that for fixed $y_1 \leq y_2$, the expression $u(x,y_1,x,y_2)$ remains nonpositive when $x \in \R$. Then $\dd|V^F|_t$-a.e., $S_t \leq 0$ and the proof of~\eqref{eq:decspr} is completed.
\end{proof}

\subsection{Derivative of the flow} We will now compute the time derivative of the flow $t \mapsto W_p^p(F_t,G_t)$ when $F$ and $G$ have the classical regularity. We first derive a nonlinear evolution equation for the pseudo-inverse function $F_t^{-1}$. Of course, the same holds for $G_t^{-1}$.

\begin{lem}\label{lem:EDPF-1}
  Assume that the uniform ellipticity condition (D3) and the regularity condition (R1) hold, that $F_0$ has a finite first order moment, that the probabilistic solution $F$ to the Cauchy problem~\eqref{eq:myPDE} with initial conditon $F_0$ has the classical regularity and that for all $0 < t_1 < t_2$, the function $(t,x) \mapsto \partial_x F_t(x)$ is bounded on $[t_1,t_2]\times\R$ . Then the pseudo-inverse function $(t,u) \mapsto F_t^{-1}(u)$ is $C^{1,2}$ on $(0,+\infty) \times (0,1)$ and satisfies
  \begin{equation}\label{eq:EDPF-1}
    \partial_t F_t^{-1}(u) = b(u) - \frac{1}{2}\partial_u\left(\frac{a(u)}{\partial_u F_t^{-1}(u)}\right).
  \end{equation}
\end{lem}
\begin{proof}
  On the one hand, since $F$ is a classical solution to~\eqref{eq:myPDE}, then 
  \begin{equation}\label{eq:solcla}
    \partial_t F_t(x) = \frac{1}{2}\left(a'(F_t(x))(\partial_x F_t(x))^2 + a(F_t(x))\partial_x^2 F_t(x)\right) - b(F_t(x))\partial_x F_t(x).
  \end{equation}
  
  On the other hand, the lower bound in the Aronson estimate~\eqref{eq:aronson} allows to apply the implicit functions theorem to $(t,x,u) \mapsto F_t(x)-u$ and deduce that $(t,u) \mapsto F_t^{-1}(u)$ is $C^{1,2}$ on $(0,+\infty) \times \R$. Besides,  
  \begin{equation}\label{eq:duF-1}
    \partial_u F_t^{-1}(u) = \frac{1}{\partial_x F_t(F_t^{-1}(u))} \qquad \text{and} \qquad \partial_u^2 F_t^{-1}(u) = -\frac{\partial_x^2 F_t(F_t^{-1}(u))}{[\partial_x F_t(F_t^{-1}(u))]^3}.
  \end{equation}
  
  For all $t > 0$, since $F_t$ is a continuous function then $F_t(F_t^{-1}(u)) = u$, so that derivating with respect to $t$ yields $0 = \partial_t F_t(F_t^{-1}(u)) + \partial_x F_t(F_t^{-1}(u))\partial_t F_t^{-1}(u)$. Using~\eqref{eq:solcla} and~\eqref{eq:duF-1}, we write
  \begin{equation*}
    \partial_t F_t^{-1}(u) = b(u) - \frac{1}{2}\left(\frac{a'(u)}{\partial_u F_t^{-1}(u)} - \frac{a(u)\partial_u^2 F_t^{-1}(u)}{[\partial_u F_t^{-1}(u)]^2}\right),
  \end{equation*}
  from which we deduce~\eqref{eq:EDPF-1}.
\end{proof}

For all function $f : \R \to [0,1]$, we define the function $\ta(f,\cdot) : \R \to [0,1]$ by
\begin{equation*}
  \ta(f,x) := \ind{x \geq 0} (1-f(x)) + \ind{x \leq 0} f(x).
\end{equation*}

Then the main result of this subsection is the following.

\begin{prop}\label{prop:quant}
  Assume that the conditions of Lemma~\ref{lem:EDPF-1} are satisfied with both $F$ and $G$. Let $p \geq 2$ such that $W_p(F_0,G_0) < +\infty$ and $|x|^{p-1}(\ta(F_0,x) + \ta(G_0,x)) \to 0$ when $x \to \pm\infty$. Then for all $0 < t_1 < t_2$,
  \begin{equation*}
    \begin{split}
      & W_p^p(F_{t_2},G_{t_2}) - W_p^p(F_{t_1},G_{t_1})\\
      & \qquad = -\frac{p(p-1)}{2}\int_{t_1}^{t_2} \int_0^1 a(u)|F_t^{-1}(u)-G_t^{-1}(u)|^{p-2}\frac{\left(\partial_u F_t^{-1}(u) - \partial_u G_t^{-1}(u)\right)^2}{\partial_u F_t^{-1}(u)\partial_u G_t^{-1}(u)}\dd u\dd t.
    \end{split}
  \end{equation*} 
\end{prop}
\begin{proof}
  See Appendix~\ref{app:quant}.
\end{proof}

\begin{rk}\label{rk:marchepas}
  A straightforward strategy to prove Proposition~\ref{prop:quant} is the following: for $\epsilon_1, \epsilon_2 \in (0,1/2)$, let
  \begin{equation*}
    W_{p,\epsilon_1,\epsilon_2}^p(F_t,G_t) := \int_{\epsilon_1}^{1-\epsilon_2} |F_t^{-1}(u) - G_t^{-1}(u)|^p\dd u,
  \end{equation*}
  then certainly $W_{p,\epsilon_1,\epsilon_2}^p(F_t,G_t) \to W_p^p(F_t,G_t)$ when $\epsilon_1, \epsilon_2 \to 0$. Now,
  \begin{equation*}
    \begin{aligned}
      & \frac{\dd}{\dd t} W_{p,\epsilon_1,\epsilon_2}^p(F_t,G_t)\\
      & \qquad = p \int_{\epsilon_1}^{1-\epsilon_2} (F_t^{-1}(u)-G_t^{-1}(u))|F_t^{-1}(u)-G_t^{-1}(u)|^{p-2} \left(\partial_t F_t^{-1}(u) - \partial_t G_t^{-1}(u)\right)\dd u\\
      & \qquad = \frac{p}{2}\int_{\epsilon_1}^{1-\epsilon_2} (F_t^{-1}(u)-G_t^{-1}(u))|F_t^{-1}(u)-G_t^{-1}(u)|^{p-2} \partial_u\left(\frac{a(u)}{\partial_u G_t^{-1}(u)} - \frac{a(u)}{\partial_u F_t^{-1}(u)}\right)\dd u\\
      & \qquad = \frac{p}{2} \left[a(u)(F_t^{-1}(u)-G_t^{-1}(u))|F_t^{-1}(u)-G_t^{-1}(u)|^{p-2}\left(\frac{1}{\partial_u G_t^{-1}(u)} - \frac{1}{\partial_u F_t^{-1}(u)}\right)\right]_{\epsilon_1}^{1-\epsilon_2}\\
      & \qquad \quad - \frac{p(p-1)}{2}\int_{\epsilon_1}^{1-\epsilon_2} a(u)|F_t^{-1}(u)-G_t^{-1}(u)|^{p-2} \frac{\left(\partial_u F_t^{-1}(u)-\partial_u G_t^{-1}(u)\right)^2}{\partial_u F_t^{-1}(u)\partial_u G_t^{-1}(u)}\dd u,
    \end{aligned}
  \end{equation*}
  where we have used~\eqref{eq:EDPF-1} at the second line and integrated by parts in the last line. Hence, Proposition~\ref{prop:quant} holds as soon as the boundary terms vanish, i.e.
  \begin{equation*}
    \liminf_{\epsilon_1, \epsilon_2 \to 0} \int_{t_1}^{t_2} \left[a(F_t^{-1}-G_t^{-1})|F_t^{-1}-G_t^{-1}|^{p-2}\left(\frac{1}{\partial_u G_t^{-1}} - \frac{1}{\partial_u F_t^{-1}}\right)\right]_{\epsilon_1}^{1-\epsilon_2} \dd t = 0.
  \end{equation*}
  However, we were not able to provide a rigorous account of this statement. In Appendix~\ref{app:quant}, we use a different expression of $W_p^p(F_t,G_t)$ in terms of $F_t$, $G_t$ to compute the time derivative of the flow.
\end{rk}

\begin{rk}\label{rk:carrillo}
  Although the setting is different, the result of Proposition~\ref{prop:contr} obtained by a probabilistic approximation is comparable to the result of Carrillo, Di Francesco and Lattanzio~\cite[Theorem~5.1]{cdl}, the proof of which relies on the deterministic operator splitting method. In Lemma~\ref{lem:EDPF-1}, the nonlinear evolution equation for the pseudo-inverse of the solution to the Cauchy problem generalizes in a rigorous way the proposed extensions of the work by Carrillo and Toscani~\cite[Section~3]{ct}. In Carrillo, Gualdani and Toscani~\cite{cgt}, the time derivative of the flow of the Wasserstein distance between two solutions is computed by the method described in Remark~\ref{rk:marchepas} for the case of compactly supported solutions at all times, therefore the boundary terms necessarily vanish in the integration by parts.
  
  The same method is also applied by Alfonsi, Jourdain and Kohatsu-Higa~\cite{euler}, where the authors get rid of the boundary terms using Gaussian estimates on the density. As Lemma~\ref{lem:aronson} shows, under the uniform ellipticity condition (D3), such estimates still hold in our case as soon as the tails of $F_0$ or $G_0$ are not heavier than Gaussian. Since we are willing to use Proposition~\ref{prop:quant} to compare $F_t$ with the stationary solution $F_{\infty}$, we would therefore need the tails of $F_{\infty}$ not to be heavier than Gaussian. But according to Remark~\ref{rk:poincare} below, under the condition (D3), the tails of $F_{\infty}$ cannot be lighter than exponential.
\end{rk}

\section{Convergence to equilibrium}\label{s:conv}

This section is divided into three parts. In Subsection~\ref{ss:stat}, we solve the stationary equation. In Subsection~\ref{ss:equi}, we prove the convergence of solutions to stationary solutions. In Subsection~\ref{ss:rate}, we discuss the (lack of) rate of convergence to equilibrium.

\subsection{The stationary equation}\label{ss:stat} We recall that the {\em stationary equation} is the following:
\begin{equation}\label{eq:stat}
  \frac{1}{2} \partial^2_x \big(A(F_{\infty}(x))\big) - \partial_x \big(B(F_{\infty}(x))\big) = 0
\end{equation} 
As mentionned in the introduction, the stationary solutions for the Cauchy problem are the cumulative distribution functions, with a finite first order moment, solving~\eqref{eq:stat} in the sense of distributions. In Proposition~\ref{prop:stat}, we solve the stationary equation, and we give a criterion for integrability in Corollary~\ref{cor:integr}.

\begin{prop}\label{prop:stat}
  Under the nondegeneracy condition (D1), a necessary and sufficient condition for the existence of cumulative distribution functions solving the stationary equation is $B(1)=0$, $B(u) \geq 0$ and the local integrability of the function $a/2B$ on $(0,1)$. Then all the solutions are continuous.
  
  If in addition $B(u) > 0$ for all $u \in (0,1)$, which corresponds to the equilibrium condition (E1), then $F_{\infty}$ is a solution if and only if there exists $\bar{x} \in \R$ such that for all $x \in \R$, $F_{\infty}(x) = \Psi^{-1}(x + \bar{x})$, where the function $\Psi$ is defined by
  \begin{equation}\label{eq:Psi}
    \forall u \in (0,1), \qquad \Psi(u) := \int_{1/2}^u \frac{a(v)}{2B(v)}\dd v.
  \end{equation}
  In this case, $\bar{x}=\Psi(F_{\infty}(0))$.
\end{prop}
\begin{proof}[Proof of Proposition~\ref{prop:stat}]
  We first prove the necessary condition. Let $F$ be a cumulative distribution function (we shall write $F$ instead of $F_{\infty}$ in the proof), solving~\eqref{eq:stat} in the sense of distributions. Then there exists $c \in \R$ such that the function $x \mapsto (1/2)A(F(x))$ is absolutely continuous with respect to the Lebesgue measure, with density $B(F(x)) + c$. Since by the condition (D1), $A$ is increasing, then $F$ is continuous. Hence, $B(F(x))+c$ is a nonnegative, continuous and integrable function, so that taking the limit $x \to -\infty$ yields $B(0)+c=c=0$. We deduce $B(u) \geq 0$ and $B(1)=0$ by taking the limit $x \to +\infty$.
  
  It remains to prove that $a(u)/2B(u)$ is locally integrable in $(0,1)$. We use the convention that $a(u)/2B(u)=0$ when $a(u)=B(u)=0$, and we define $U := \{u \in (0,1) : B(u) = 0 \text{ and } a(u) \not= 0\}$. Let $0 < \alpha < \beta < 1$ and $x_- := F^{-1}(\alpha)$, $x_+ := F^{-1}(\beta)$. We denote by $\dd F(x)$ the Stieltjes measure associated with the continuous function $F$ of finite variation. Then by the chain rule formula~\cite[(4.6) p.~6]{revuz}, the Radon measure $(1/2)a(F(x))\dd F(x)$ has the density $B(F(x))$ with respect to the Lebesgue measure. Therefore $\dd F(x)$-a.e., $F(x) \not\in U$. By the change of variable formula~\cite[(4.9) p.~8]{revuz},
  \begin{equation*}
    \int_{x_-}^{x_+} \frac{a(F(x))}{2B(F(x))}\ind{F(x) \not\in U}\dd F(x) = \int_{x_-}^{x_+} \frac{a(F(x))}{2B(F(x))}\dd F(x) = \int_{\alpha}^{\beta} \frac{a(u)}{2B(u)}\dd u,
  \end{equation*}
  and the left-hand side is bounded by $x_+-x_- < +\infty$.
  
  We now prove that the condition is sufficient. For $u \in (0,1)$, let us define $\Psi(u)$ as in~\eqref{eq:Psi}. Then $\Psi$ is absolutely continuous and since $A$ is increasing, so is $\Psi$. Thus $\Psi^{-1}$ is continuous, with finite variation. In case $\Psi$ has a finite limit in $1$ (resp. $0$), we write $\Psi^{-1}(x) = 1$ (resp. $0$) for $x \in [\Psi(1),+\infty)$ (resp. $(-\infty,\Psi(0)]$) so that $\Psi^{-1}$ is a cumulative distribution function on $\R$. We now check that $\Psi^{-1}$ is a solution, in the sense of distributions, of the stationary equation. Let $\phi\in\Cc(\R)$ be a test function. Then by the integration by parts formula~\cite[(4.5) p.~6]{revuz} and the chain rule formula applied to $A(\Psi^{-1}(x))$,
  \begin{equation*}
    -\frac{1}{2}\int_{\R} \phi'(x) A(\Psi^{-1}(x))\dd x = \frac{1}{2}\int_{\R} \phi(x) a(\Psi^{-1}(x))\dd \Psi^{-1}(x).
  \end{equation*}
  By the change of variable formula and the definition of $\Psi$,
  \begin{equation*}
    \frac{1}{2}\int_{\R} \phi(x) a(\Psi^{-1}(x))\dd \Psi^{-1}(x) = \frac{1}{2}\int_0^1 \phi(\Psi(u)) a(u)\dd u = \int_0^1 \phi(\Psi(u)) B(u) \Psi'(u)\dd u,
  \end{equation*}
  and performing a new change of variables in the last member above yields
  \begin{equation*}
    -\frac{1}{2}\int_{\R} \phi'(x) A(\Psi^{-1}(x))\dd x = \int_{\R} \phi(x) B(\Psi^{-1}(x))\dd x ,
  \end{equation*}
  i.e. $\Psi^{-1}$ is a solution in the sense of distributions of the stationary equation. 
  
  We finally assume that $B(u) > 0$ for all $u \in (0,1)$ and prove that $F$ is a solution if and only if it is a translation of $\Psi^{-1}$. If $F$ is a solution, since $B(u) > 0$ then $\Psi$ is $C^1$ on $(0,1)$ so that the chain rule formula gives
  \begin{equation*}
    \dd (\Psi(F(x))) = \Psi'(F(x))\dd F(x) = \frac{a(F(x))}{2B(F(x))}\dd F(x) = \dd x,
  \end{equation*}
  where the last equality holds due to $B(u) > 0$. Then $\Psi(F(x)) - \Psi(F(0)) = x$ and $F(x)=\Psi^{-1}(x + \bar{x})$ with $\bar{x}=\Psi(F(0))$. Reciprocally it is immediate that all the translations of $\Psi^{-1}$ solve the stationary equation.
\end{proof}

\begin{rk}
  When the condition that $B(u) > 0$ on $(0,1)$ is not fulfilled, then one can exhibit solutions that are not translations of each other. For instance, let $a(u) = u(1-u)|u-1/2|^{3/2}$ and $B(u) = u(1-u)(u-1/2)^2$. Then $a$ is continuous on $[0,1]$, its antiderivative satisfies (D1), $B$ is $C^1$ on $[0,1]$, $B(0)=B(1)=0$ and $B(u) \geq 0$. Besides, $\Psi(u) = \sgn(u-1/2)|u-1/2|^{1/2}$. For all $h \geq 0$, let us define
  \begin{equation*}
    F_{\infty,h}(x) := \left\{\begin{aligned}
      & 0 & x < -1/\sqrt{2},\\
      & 1/2 - x^2 & -1/\sqrt{2} \leq x < 0,\\
      & 1/2 & 0 \leq x < h,\\
      & 1/2 + (x-h)^2 & h \leq x < h + 1/\sqrt{2},\\
      & 1 & x \geq h + 1/\sqrt{2}.
    \end{aligned}\right.
  \end{equation*}
  Then $F_{\infty,0} = \Psi^{-1}$ and for all $h > 0$, $F_{\infty,h}$ solves the stationary equation although it is not a translation of $\Psi^{-1}$.
\end{rk}

In order to apply the results of Section~\ref{s:contr}, we need criteria ensuring the existence of a first order moment as well as the classical regularity for a stationary solution $F_{\infty}$. They come as corollaries to Proposition~\ref{prop:stat}.

\begin{cor}\label{cor:regul}
  Under the nondegeneracy condition (D2), the regularity condition (R1) and the equilibrium condition (E1), all the stationary solutions $F_{\infty}$ are $C^2$ on $\R$.
\end{cor}
\begin{proof}
  By Proposition~\ref{prop:stat} and the condition (E1), it is enough to prove that $\Psi^{-1}$ is $C^2$ on $\R$, which follows from the inverse function theorem, since $\Psi$ is $C^2$ on $(0,1)$ by (R1) and (E1), and $\Psi'(u)>0$ for all $u \in (0,1)$ by (D2). 
\end{proof}

\begin{cor}\label{cor:integr}
  Under the nondegeneracy condition (D1) and the equilibrium condition (E1), the solutions to the stationary equation have a finite first order moment if and only if (E2) holds.
\end{cor}
\begin{proof}
  According to Proposition~\ref{prop:stat}, it is enough to prove the statement for the solution $\Psi^{-1}$, the first order moment of which is given by
  \begin{equation*}
    \int_{1/2}^1 \Psi(u)\dd u - \int_0^{1/2} \Psi(u)\dd u = \int_0^{1/2} \frac{a(u) u}{2B(u)}\dd u + \int_{1/2}^1 \frac{a(u)(1-u)}{2B(u)}\dd u
  \end{equation*}
  by the Fubini-Tonelli theorem. The finiteness of the right-hand side is the condition (E2).
\end{proof}

\begin{rk}\label{rk:poincare}
  In the case of the viscous conservation law and under the condition (E1), it is sufficient that $b(0)>0$ and $b(1)<0$ for (E2) to hold, and in this case the probability distributions derivated from the stationary solutions have exponential tails and satisfy a Poincaré inequality (see~\cite[Lemma~2.1]{jm}). In the general case of the stationary equation~\eqref{eq:stat}, and under the equilibrium condition (E1), these results extend as follows.
  \begin{itemize}
    \item Under the uniform ellipticity condition (D3), if $b(0)>0$ and $b(1)<0$ then it is clear from the expression of $\Psi$ that the stationary solutions still have exponential tails, and they consequently satisfy a Poincaré inequality. If $b(0)=0$ (resp. $b(1)=0$), then the left (resp. the right) tail is heavier than exponential.
    
    \item Under the nondegeneracy condition (D1), as soon as the cumulative distribution function $\Psi^{-1}$ admits a positive density $p$, then it satisfies a Poincaré inequality if and only if it satisfies the Hardy criterion (see~\cite[Theorem~6.2.2, p.~99]{logsob}), namely
  \begin{equation*}
    \sup_{x \geq 0} \int_x^{+\infty} p(y)\dd y \int_0^x \frac{\dd y}{p(y)} < +\infty \text{,}\quad \sup_{x \leq 0} \int_{-\infty}^x p(y)\dd y \int_x^0 \frac{\dd y}{p(y)} < +\infty .
  \end{equation*}
  Letting $y = \Psi(v)$, we rewrite 
  \begin{equation*}
    \int_0^x \frac{\dd y}{p(y)} = \int_{1/2}^{\Psi^{-1}(x)}(\Psi'(v))^2\dd v,
  \end{equation*}
  so that that the stationary solutions satisfy a Poincaré inequality if and only if
  \begin{equation*}
    \sup_{u \geq 1/2} (1-u)\int_{1/2}^u \left(\frac{a(v)}{2B(v)}\right)^2\dd v < +\infty \text{,}\quad \sup_{u \leq 1/2} u\int_u^{1/2} \left(\frac{a(v)}{2B(v)}\right)^2\dd v < +\infty.
  \end{equation*}
  \end{itemize}
\end{rk}

\subsection{Convergence in Wasserstein distance}\label{ss:equi} We now state the main result of the article, namely the convergence to equilibrium of the probabilistic solutions in Wasserstein distance.

\begin{theo}\label{theo:equi}
  Let us assume that:
  \begin{itemize}
    \item the coefficients of the Cauchy problem~\eqref{eq:myPDE} satisfy the uniform ellipticity condition (D3), the regularity condition (R2) and the equlibrium conditions (E1) and (E2);
    \item the probability distribution $m$ has a finite first order moment;
    \item $W_2(H*m,\Psi^{-1}) < +\infty$.
  \end{itemize}
  Let $F$ be the probabilistic solution the Cauchy problem~\eqref{eq:myPDE} with initial condition $H*m$. Then there exists a unique stationary solution $F_{\infty}$ such that $F_0$ and $F_{\infty}$ have the same expectation, and for all $p \geq 2$ such that $W_p(H*m,\Psi^{-1}) < +\infty$,
  \begin{equation*}
    \forall 1 \leq q < p, \qquad \lim_{t\to+\infty} W_q(F_t, F_{\infty}) = 0.
  \end{equation*}
\end{theo}
\begin{proof}
  The proof is in 6 steps. 
  
  \sk
  {\bf Step 1.} We first prove the existence and uniqueness of a stationary solution $F_{\infty}$ such that $F_0$ and $F_{\infty}$ have the same expectation. By the condition (E2), the integral $\int_0^1 \Psi(u)\dd u$ is defined. Owing to the condition (E1) and according to Proposition~\ref{prop:stat}, the stationary solutions are the functions of the form $F_{\infty}(x) = \Psi^{-1}(x+\bar{x})$, $\bar{x}\in\R$. By Corollary~\ref{cor:integr}, the expectation of such a cumulative distribution function exists and it is given by
  \begin{equation*}
    \int_0^1 F_{\infty}^{-1}(u) \dd u = \int_0^1 \Psi(u)\dd u - \bar{x}.
  \end{equation*}
  Thus, the unique stationary solution with the same expectation as $F_0$ is given by $F_{\infty}(x) = \Psi^{-1}(x+\bar{x})$, with $\bar{x} = \int_0^1(\Psi(u) - F_0^{-1}(u))\dd u$.   
  
  \sk
  {\bf Step 2.} We now introduce a smooth approximation of the initial condition $F_0$ in order to use Lemma~\ref{lem:regul}. Let $\zeta$ be a $C^{\infty}$ probability density on $\R$ with compact support and such that $\int_{\R} x \zeta(x)\dd x = 0$. For all $\alpha > 0$ and $x \in \R$, we define $\zeta_{\alpha}(x) := \alpha^{-1}\zeta(\alpha^{-1}x)$ and $F_0^{\alpha}(x) := F_0 * \zeta_{\alpha}(x)$. Then $F_0^{\alpha}$ is $C^{\infty}$ on $\R$ and it is the cumulative distribution function of $X_0 + \alpha Z$, where $X_0$ has distribution $m$, $Z$ has density $\zeta$ and $X_0$ and $Z$ are independent. As a consequence, $F_0^{\alpha}$ has a finite first order moment and it has the same expectation as $F_0$ and $F_{\infty}$ due to Step~1. By Lemma~\ref{lem:regul}, the probabilistic solution $F^{\alpha}$ to the Cauchy problem~\eqref{eq:myPDE} with initial condition $F_0^{\alpha}$ has the classical regularity and for all finite $T>0$, it belongs to $C^{1,2}_{\Lip}([0,T]\times\R)$.
  
  Using the obvious coupling $(X_0, X_0+\alpha Z)$ of $F_0$ and $F_0^{\alpha}$, we note that for all $q \geq 1$, $W_q(F_0, F_0^{\alpha}) \leq \alpha \Exp(|Z|^q)^{1/q} < +\infty$. This leads to the following remarks:
  \begin{itemize}
    \item As $F_{\infty}$ is a translation of $\Psi^{-1}$, by the triangle inequality, $W_q(H*m,\Psi^{-1})$ and $W_q(F_0^{\alpha},F_{\infty})$ are simultaneously finite or infinite. 
    \item Using the triangle inequality again and Proposition~\ref{prop:contr}, we get 
    \begin{equation*}
      W_q(F_t,F_{\infty}) \leq W_q(F_t, F_t^{\alpha}) + W_q(F_t^{\alpha},F_{\infty}) \leq \alpha \Exp(|Z|^q)^{1/q} + W_q(F_t^{\alpha},F_{\infty}).
    \end{equation*}
     Hence, as soon as, for all $\alpha > 0$, $\limsup_{t \to +\infty} W_q(F_t^{\alpha}, F_{\infty}) = 0$, then taking $\alpha$ arbitrarily small yields $\lim_{t \to +\infty} W_q(F_t, F_{\infty}) = 0$.
  \end{itemize}
  
  We now fix $\alpha > 0$. The remaining steps are dedicated to the proof of the fact that, for all $p \geq 2$ such that $W_p(H*m, \Psi^{-1}) < +\infty$, for all $1 \leq q < p$, $\lim_{t \to +\infty} W_q(F_t^{\alpha}, F_{\infty}) = 0$.
  
  \sk
  {\bf Step 3.} We prove that for all $t \geq 0$, the expectation of $F_t^{\alpha}$ remains constant. By Corollary~\ref{cor:nlmp}, for all $t \geq 0$, $F_t^{\alpha}$ is the marginal cumulative distribution function of the nonlinear diffusion process $X^{\alpha}$ solution to~\eqref{eq:nldp} with initial condition having cumulative distribution function $F_0^{\alpha}$. Since $\sigma$ and $b$ are bounded,
  \begin{equation*}
    \forall t \geq 0, \qquad \Exp(X_t^{\alpha}) = \Exp(X_0^{\alpha}) + \int_0^t \Exp[b(F_s^{\alpha}(X_s^{\alpha}))]\dd s.
  \end{equation*}
  But for all $s>0$, $F_s^{\alpha}$ is continuous so that $\Exp[b(F_s^{\alpha}(X_s^{\alpha}))] = \int_0^1 b(u)\dd u = B(1)$. By (E1), we conclude that $\Exp(X_t^{\alpha})=\Exp(X_0^{\alpha})$.
  
  \sk
  {\bf Step 4.} We now describe the evolution of the Wasserstein distance $W_2(F_t^{\alpha},F_{\infty})$. We are willing to use Proposition~\ref{prop:quant}, therefore we need to check that $F^{\alpha}$ and $F_{\infty}$ satisfy the assumptions of Lemma~\ref{lem:EDPF-1}. It is the case for $F^{\alpha}$ thanks to Lemma~\ref{lem:regul}. The stationary solution $F_{\infty}$ has a finite first order moment owing to the condition (E2) and Corollary~\ref{cor:integr}, it is $C^2$ on $\R$ by the condition (R2) and Corollary~\ref{cor:regul}, and from the definition of $\Psi^{-1}$ and condition (D3) it follows that the derivative of $F_{\infty}$ is bounded by $2||B||_{\infty}/\ua$.
  
  Moreover, by the assumption that $W_2(H*m,\Psi^{-1}) < +\infty$ and Step~2, $W_2(F_0^{\alpha}, F_{\infty}) < +\infty$; and since both $F_0^{\alpha}$ and $F_{\infty}$ have a finite first order moment, $|x|(\ta(F_0^{\alpha},x) + \ta(F_{\infty},x))$ vanishes when $x \to \pm\infty$. Thus, Proposition~\ref{prop:quant} applies to $F^{\alpha}$ and $F_{\infty}$ with $p=2$ and yields, for all $0 < t_1 < t_2$,
  \begin{equation*}
    W_2^2(F_{t_2}^{\alpha},F_{\infty}) - W_2^2(F_{t_1}^{\alpha},F_{\infty}) = - \int_{t_1}^{t_2} \int_0^1 a(u)\frac{\left(\partial_u (F_t^{\alpha})^{-1}(u) - \partial_u F_{\infty}^{-1}(u)\right)^2}{\partial_u (F_t^{\alpha})^{-1}(u)\partial_u F_{\infty}^{-1}(u)}\dd u \dd t \leq 0.
  \end{equation*} 
  
  Using the uniform ellipticity condition (D3), we can then assert that
  \begin{equation*}
    \liminf_{t \to +\infty} \int_0^1 \frac{\left(\partial_u (F_t^{\alpha})^{-1}(u) - \partial_u F_{\infty}^{-1}(u)\right)^2}{\partial_u (F_t^{\alpha})^{-1}(u)\partial_u F_{\infty}^{-1}(u)}\dd u = 0,
  \end{equation*}
  and extract a sequence $(t_n)_{n \geq 1}$ growing to $+\infty$ such that the integral above goes to $0$ along $(t_n)_{n \geq 1}$. Let us prove that for all $u \in (0,1)$,
  \begin{equation}\label{eq:step3}
    \lim_{n \to +\infty} \int_{1/2}^u |\partial_u (F_{t_n}^{\alpha})^{-1}(v) - \partial_u F_{\infty}^{-1}(v)|\dd v = 0.
  \end{equation}
  Let $0 < \epsilon < 1/2$. By the Cauchy-Schwarz inequality and the condition (E1),
  \begin{equation}\label{eq:step3:2}
    \begin{split}
      & \int_{\epsilon}^{1-\epsilon} |\partial_u (F_{t_n}^{\alpha})^{-1} - \partial_u F_{\infty}^{-1}|\dd v \leq \left(\int_{\epsilon}^{1-\epsilon}\partial_u (F_{t_n}^{\alpha})^{-1}\partial_u F_{\infty}^{-1}\dd v\int_{\epsilon}^{1-\epsilon}\frac{\left(\partial_u (F_{t_n}^{\alpha})^{-1} - \partial_u F_{\infty}^{-1}\right)^2}{\partial_u (F_{t_n}^{\alpha})^{-1}\partial_u F_{\infty}^{-1}}\dd v\right)^{1/2}\\
      & \qquad \leq \left(\sup_{v \in [\epsilon, 1-\epsilon]} \frac{a(v)}{2B(v)}\int_{\epsilon}^{1-\epsilon}\partial_u (F_{t_n}^{\alpha})^{-1}\dd v\int_{\epsilon}^{1-\epsilon}\frac{\left(\partial_u (F_{t_n}^{\alpha})^{-1} - \partial_u F_{\infty}^{-1}\right)^2}{\partial_u (F_{t_n}^{\alpha})^{-1}\partial_u F_{\infty}^{-1}}\dd v\right)^{1/2}.
    \end{split}
  \end{equation}
  The first integral can be bounded uniformly in $n$ as follows:
  \begin{equation*}
    \begin{split}
      & \int_{\epsilon}^{1-\epsilon}\partial_u (F_{t_n}^{\alpha})^{-1}(v)\dd v = (F_{t_n}^{\alpha})^{-1}(1-\epsilon) - (F_{t_n}^{\alpha})^{-1}(\epsilon)\\
      & \qquad \leq \frac{2}{\epsilon} \int_0^{\epsilon/2}\left((F_{t_n}^{\alpha})^{-1}(1-\epsilon+v) - (F_{t_n}^{\alpha})^{-1}(\epsilon-v)\right)\dd v\\
      & \qquad \leq \frac{2}{\epsilon} \left(\int_0^{\epsilon/2}\left(F_{\infty}^{-1}(1-\epsilon+v) - F_{\infty}^{-1}(\epsilon-v)\right)\dd v + \int_0^1 |(F_{t_n}^{\alpha})^{-1}(v) - F_{\infty}^{-1}(v)|\dd v\right)\\
      & \qquad \leq \frac{2}{\epsilon} \left(\int_0^{\epsilon/2}\left(F_{\infty}^{-1}(1-\epsilon+v) - F_{\infty}^{-1}(\epsilon-v)\right)\dd v + \int_0^1 |(F_0^{\alpha})^{-1}(v) - F_{\infty}^{-1}(v)|\dd v\right),
    \end{split}
  \end{equation*}
  where the last inequality is due to Proposition~\ref{prop:contr}. We deduce that the right-hand side of~\eqref{eq:step3:2} goes to $0$, so that taking $\epsilon \leq u \wedge (1-u)$ yields~\eqref{eq:step3}.
  
  \sk
  {\bf Step 5.} We extract a subsequence of $(t_n)_{n \geq 1}$, that we still index by $n$ for convenience, such that $\lim_{n \to +\infty} (F_{t_n}^{\alpha})^{-1}(1/2) - F_{\infty}^{-1}(1/2) = \ell \in [-\infty,+\infty]$. Then using Step~4, for all $u \in (0,1)$ one has $(F_{t_n}^{\alpha})^{-1}(u) - F_{\infty}^{-1}(u) \to \ell$. Besides, since by Proposition~\ref{prop:contr},
  \begin{equation*}
    \sup_{t \geq 0} \int_0^1 |(F_t^{\alpha})^{-1}(u) - F_{\infty}^{-1}(u)|^2\dd u = W_2^2(F_0^{\alpha},F_{\infty}) < +\infty,
  \end{equation*}
  then the functions $(u \mapsto (F_{t_n}^{\alpha})^{-1}(u) - F_{\infty}^{-1}(u))_{n \geq 1}$ are uniformly integrable. We deduce using Step~3 that
  \begin{equation*}
    \ell = \lim_{n \to +\infty} \int_0^1 \left((F_{t_n}^{\alpha})^{-1}(u) - F_{\infty}^{-1}(u)\right) \dd u = 0.
  \end{equation*}
  
  \sk
  {\bf Step 6.} Let $p \geq 2$ such that $W_p(F_0^{\alpha},\Psi^{-1}) < +\infty$. Then by Step~2, $W_p(F_0^{\alpha},F_{\infty})<+\infty$; therefore, for all $1 \leq q < p$, the functions $(u \mapsto |(F_{t_n}^{\alpha})^{-1}(u) - F_{\infty}^{-1}(u)|^q)_{n \geq 1}$ are uniformly integrable, and using Step~5 we have
  \begin{equation*}
    \lim_{n \to +\infty} \int_0^1 |(F_{t_n}^{\alpha})^{-1}(u) - F_{\infty}^{-1}(u)|^q \dd u = 0.
  \end{equation*}
  But according to Proposition~\ref{prop:contr}, the flow $t \mapsto W_q(F_t^{\alpha},F_{\infty})$ is nonincreasing. As a consequence $\lim_{t \to +\infty} W_q(F_t^{\alpha},F_{\infty}) = 0$ and the proof is completed by virtue of Step~2.
\end{proof}

\subsection{Rate of convergence}\label{ss:rate} We first recall the result of convergence to equilibrium stated in~\cite{jm}, where $A(u) = \sigma^2 u$ with $\sigma^2 > 0$. Then it is easily checked that the conditions (R1), (E1) and (E2) are satisfied if  $B$ is $C^2$ on $[0,1]$, with $B(1)=0$, $b(0) > 0$, $b(1) < 0$ and $B(u) > 0$ on $(0,1)$. Then according to Remark~\ref{rk:poincare}, all the stationary solutions $F_{\infty}$ admit a positive density $p_{\infty}$ and satisfy a Poincaré inequality. Under these assumptions, we have the following convergence result.

\begin{lem}\cite[Lemma~2.8]{jm}\label{lem:jm2.8}
  There exist $\eta > 0$ and $c > 0$ depending on $A$ and $B$ such that for all cumulative distribution function $F_0$ with a finite first order moment, calling $F_{\infty}$ the sationary solution with the same expectation as $F_0$, as soon as $\int (F_0-F_{\infty})^2/p_{\infty} \dd x \leq \eta$ then
  \begin{equation*}
    \forall t \geq 0, \qquad \int_{\R} \frac{(F_t(x)-F_{\infty}(x))^2}{p_{\infty}(x)} \dd x \leq \frac{\exp(-ct)}{c} \int_{\R} \frac{(F_0(x)-F_{\infty}(x))^2}{p_{\infty}(x)} \dd x.
  \end{equation*}
\end{lem}

According to~\cite[Proposition~1.4]{jourdain:transport}, the quadratic Wasserstein distance between two probability distributions $\mu$ and $\nu$ on $\R$ such that $\mu$ admits a positive density $p$ satisfies the inequality
\begin{equation*}
  W_2^2(\mu,\nu) \leq 4 \int_{\R} \frac{(H*\mu(x) - H*\nu(x))^2}{p(x)} \dd x.
\end{equation*}
Hence, the convergence result of Lemma~\ref{lem:jm2.8} can be translated in terms of the Wasserstein distance.

\begin{cor}\label{cor:rate}
  Under the assumptions of Lemma~\ref{lem:jm2.8}, as soon as $\int (F_0-F_{\infty})^2/p_{\infty} \dd x$ is small enough, then $W_2(F_t, F_{\infty})$ converges to $0$ exponentially fast. 
\end{cor}

\appendix

\section{Proof of Proposition~\ref{prop:unicite-edp}}\label{app:unicite}

This appendix is dedicated to the proof of the first point of Proposition~\ref{prop:unicite-edp}, which states that there is at most one weak solution to the Cauchy problem~\eqref{eq:myPDE} in the set $\F(T)$, for all $T>0$, possibly $T=+\infty$. The proof is adaptated from Wu, Zhao, Yin and Lin~\cite[Section 3.2]{zhao-livre} as well as Liu and Wang~\cite{liu}, who provide uniqueness of bounded weak solutions to the {\em initial-boundary value} problem, namely the Cauchy problem~\eqref{eq:myPDE} in the strip $[0,T) \times (0,1)$ with boundary conditions at $x=0$ and $x=1$. At an intuitive level, one can see our restriction to the set $\F(T)$ as some boundary conditions at $x=-\infty$ and $x=+\infty$. 

We shall follow the so-called Holmgren's approach which consists in turning the proof of uniqueness for~\eqref{eq:myPDE} into a proof of existence for an {\em adjoint problem}. Recall that we make the following nondegeneracy assumption:
\begin{itemize}
  \item[(D1)] The function $A$ is increasing.
\end{itemize}

Let $T>0$, possibly $T=+\infty$, and let $F^1$, $F^2 \in \F(T)$ such that for all $g \in \Cc([0,T)\times\R)$, both $F^1$ and $F^2$ satisfy~\eqref{eq:weak}. Then, for all $t \in [0,T)$, the function $F^2_t - F^1_t$ is integrable on $\R$ and the function $(s,x) \mapsto F^2_s(x) - F^1_s(x)$ is integrable on $Q_t := (0,t) \times \R$. Therefore, for all $t \in [0,T)$ and for all $g \in \Cc([0, T) \times \R)$,~\eqref{eq:weak} yields
\begin{equation}\label{eq:F2-F1}
  \int_{Q_t} (F^2_s-F^1_s) \left\{\frac{1}{2} \tA\partial^2_x g + \tB\partial_x g + \partial_s g\right\} \dd s\dd x = \int_{\R} (F^2_t(x)-F^1_t(x))g(t,x)\dd x;
\end{equation}
where
\begin{equation*}
  \tA(s,x) = \int_0^1 a\left((1-\theta)F^1_s(x) + \theta F^2_s(x)\right)\dd \theta,
\end{equation*}
and
\begin{equation*}
  \tB(s,x) = \int_0^1 b\left((1-\theta)F^1_s(x) + \theta F^2_s(x)\right)\dd \theta.
\end{equation*}

\begin{rk}\label{rk:C12}
  For all $t \in [0,T)$, by a classical regularization argument the integral equality~\eqref{eq:F2-F1} holds true for all function $g$ in the space $C^{1,2}_{\mathrm{b}}([0,t] \times \R)$ of real-valued $C^{1,2}$ functions bounded together with their derivatives.
\end{rk}

Let $f \in \Cc([0, T) \times \R)$. Then there exists $t\in[0,T)$ such that $\Supp f \subset [0,t) \times \R$. Let us introduce the  {\em adjoint problem} to~\eqref{eq:myPDE} as
\begin{equation}\label{eq:adj}
  \left\{\begin{aligned}
    & \frac{1}{2}\tA\partial^2_x g + \tB\partial_x g + \partial_s g = f &\qquad (s,x) \in [0,t) \times \R,\\
    & g(t,x) = 0 &\qquad x \in \R.
  \end{aligned}\right. 
\end{equation}

The coefficients $\tA$ and $\tB$ may not be smooth enough to allow the adjoint problem to admit classical solutions. Therefore we introduce a suitable approximation of~\eqref{eq:adj}. For small $\delta, \eta > 0$, let 
\begin{equation}\label{eq:FG}
  \begin{split}
    G_{\delta} & := \{(s,x) \in [0,t] \times \R : |F_s^1(x) - F_s^2(x)| < \delta\},\\
    F_{\delta} & := \{(s,x) \in [0,t] \times \R : |F_s^1(x) - F_s^2(x)| \geq \delta\},
  \end{split}
\end{equation}
and let us define
\begin{equation*}
  \lambda_{\eta}^{\delta}(s,x) = \left\{\begin{aligned}
    & 0 & \text{on $G_{\delta}$},\\
    & \left[\frac{1}{2}(\eta+\tA(s,x))\right]^{-1/2}\tB(s,x) & \text{on $F_{\delta}$}.
  \end{aligned}\right. 
\end{equation*}
Since $A$ is increasing and $F^1,F^2$ are bounded, there exist $L(\delta) > 0$ and $K(\delta) > 0$ independent of $\eta$ such that
\begin{equation*}
  \begin{aligned}
    &\tA(s,x) \geq L(\delta) &\qquad (s,x) \in F_{\delta},\\
    &|\lambda_{\eta}^{\delta}(s,x)| \leq K(\delta) &\qquad (s,x) \in [0,t] \times \R.
  \end{aligned}
\end{equation*}

Let $\xi$ be a $C^{\infty}$ probability density on $\R^2$ such that $\Supp \xi \subset [-1, 1] \times [-1,1]$. For all $\epsilon > 0$, let $\xi_{\epsilon} := \epsilon^{-2}\xi(\epsilon^{-1}s,\epsilon^{-1}x)$ and define $\tA_{\epsilon} = \tA * \xi_{\epsilon}$ and $\lambda_{\eta,\epsilon}^{\delta} = \lambda_{\eta}^{\delta} * \xi_{\epsilon}$. Then $\tA_{\epsilon}$ and $\lambda_{\eta,\epsilon}^{\delta}$ are $C^{\infty}$ functions and all their derivatives are bounded on $[0,t] \times \R$. Besides,
\begin{equation}\label{eq:est}
  \begin{aligned}
    &\lim_{\epsilon\to 0}\tA_{\epsilon}(s,x) = \tA(s,x) &\qquad \text{a.e. in $[0,t]\times\R$},\\
    &\lim_{\epsilon\to 0}\lambda_{\eta,\epsilon}^{\delta}(s,x) = \lambda_{\eta}^{\delta}(s,x) &\qquad \text{a.e. in $[0,t]\times\R$},\\
    &\tA_{\epsilon}(s,x) \leq C &\qquad (s,x) \in [0,t]\times\R,\\
    &|\lambda_{\eta,\epsilon}^{\delta}(s,x)| \leq K(\delta) &\qquad (s,x) \in [0,t]\times\R,
  \end{aligned}
\end{equation}
where $C$ refers to a positive constant independent of $\epsilon$, $\delta$ and $\eta$, and $K(\delta)$ refers to a positive constant depending only on $\delta$. In the sequel, the values of $C$ and $K(\delta)$ can change from one line to another.

We finally define
\begin{equation*}
  \tB_{\eta,\epsilon}^{\delta}(s,x) = \lambda_{\eta,\epsilon}^{\delta}(s,x)\left[\frac{1}{2}(\eta+\tA_{\epsilon}(s,x))\right]^{1/2},
\end{equation*}
and emphasize the fact that 
\begin{equation}\label{eq:BK}
  ||\tB_{\eta,\epsilon}^{\delta}||_{\infty} \leq K(\delta).
\end{equation}

We are now able to introduce the {\em approximate adjoint problem}
\begin{equation}\label{eq:appadj}
  \left\{\begin{aligned}
    & \frac{1}{2}(\eta+\tA_{\epsilon})\partial^2_x g + \tB_{\eta,\epsilon}^{\delta}\partial_x g + \partial_s g = f &\qquad (s,x) \in [0,t) \times \R,\\
    & g(t,x) = 0 &\qquad x \in \R.
  \end{aligned}\right. 
\end{equation}
The coefficients of the equation are bounded, globally Lipschitz continuous, the operator is uniformly parabolic, and the right-hand side $f$ is continuous and bounded. Therefore the Cauchy problem~\eqref{eq:appadj} admits a unique classical bounded solution $\g$ (see~\cite[p.~369]{karatzas}). Since the coefficients of the equation and $f$ are $C^{\infty}$ on $[0,t] \times \R$, then so is $\g$ (see~\cite[p.~263]{friedman}). Owing to the Feynman-Kac formula, $\g$ has the following probabilistic representation:
\begin{equation}\label{eq:FK}
  \forall (s,x) \in [0,t)\times\R, \qquad \g(s,x) = -\Exp\left[\int_s^t f(r,Z^{s,x}_r)\dd r\right]
\end{equation}
where, for a given standard Brownian motion $W$, $(Z^{s,x}_r)_{r \in [0,t]}$ is the unique strong solution of the stochastic differential equation
\begin{equation}\label{eq:eds}
  Z^{s,x}_r = x + \int_s^r \tB_{\eta,\epsilon}^{\delta}(u,Z^{s,x}_u)\dd u + \int_s^r (\eta+\tA_{\epsilon}(u,Z^{s,x}_u))^{1/2}\dd W_u.
\end{equation}

\begin{lem}\label{lem:g}
  The functions $\g$, $\partial_x \g$ and $\partial_x^2 \g$ are such that:
  \begin{eqnarray}
    \label{eq:supg}& &\sup_{[0,t] \times \R} |\g(s,x)| \leq C,\\
    \label{eq:intg}& &\sup_{s \in [0,t]} \int_{\R} |\g(s,x)|\dd x \leq K(\delta),\\
    \label{eq:dgL1}& &\sup_{s \in [0,t]} |\partial_x \g(s,x)| \leq \kappa(\epsilon,\delta,\eta)\exp(-x^2/\kappa(\epsilon,\delta,\eta)),\\
    \label{eq:d2gL1}& &\sup_{s \in [0,t]} |\partial_x^2 \g(s,x)| \leq \kappa(\epsilon,\delta,\eta)\exp(-x^2/\kappa(\epsilon,\delta,\eta)),
  \end{eqnarray}
  where the value of $\kappa(\epsilon,\delta,\eta)$ can change from one line to another.
  \end{lem}
\begin{proof}
  The inequality~\eqref{eq:supg} directly follows from the Feynman-Kac formula~\eqref{eq:FK}. Besides, since $f$ has a compact support in $[0,t] \times \R$, say $\Supp f \subset [0,t] \times [x_-, x_+]$, one has
  \begin{equation*}
    |\g(s,x)| \leq ||f||_{\infty}\int_s^t \Pr(Z^{s,x}_r \in [x_-,x_+]) \dd r,
  \end{equation*}
  and for $x > x_+$,
  \begin{equation*}
    \Pr(Z^{s,x}_r \in [x_-,x_+]) \leq \Pr(x-Z^{s,x}_r \geq x-x_+) \leq \frac{\Exp[(x-Z^{s,x}_r)^2]}{(x-x_+)^2}.
  \end{equation*} 
  Owing to~\eqref{eq:eds} and~\eqref{eq:BK},
  \begin{equation*}
    \Exp[(x-Z^{s,x}_r)^2] \leq 2\left((r-s)^2||\tB_{\eta,\epsilon}^{\delta}||_{\infty}^2 + (r-s)||\eta + \tA_{\epsilon}||_{\infty}\right) \leq K(\delta),
  \end{equation*}
  and similar arguments for $x < x_-$ yield~\eqref{eq:intg}.
    
  In order to prove~\eqref{eq:dgL1} and~\eqref{eq:d2gL1}, let us take the derivative with respect to $x$ of the problem~\eqref{eq:appadj}. Then the function $\partial_x \g$ is the unique classical solution of the Cauchy problem
  \begin{equation}\label{eq:pbdg}
    \left\{\begin{aligned}
      & \frac{1}{2}(\eta+\tA_{\epsilon})\partial^2_x g^1 + \left(\frac{1}{2}\partial_x\tA_{\epsilon} + \tB_{\eta,\epsilon}^{\delta}\right)\partial_x g^1 + \partial_x \tB_{\eta,\epsilon}^{\delta} g^1 + \partial_s g^1 = \partial_x f,\\
      & g^1(t,x) = 0,
    \end{aligned}\right. 
  \end{equation}
  and the Feynman-Kac formula now writes
  \begin{equation}\label{eq:FK1}
    \begin{split}
      \partial_x \g(s,x) &= -\Exp\left[\int_s^t \partial_x f(r,Z^{1,s,x}_r)\exp\left(\int_s^r \partial_x \tB_{\eta,\epsilon}^{\delta}(u,Z^{1,s,x}_u)\dd u\right)\dd r\right]\\
      &=  -\int_s^t\int_{\R} \partial_x f(r,z)G^1(s,x;r,z)\dd z\dd r
    \end{split}
  \end{equation}
  where $(Z^{1,s,x}_r)_{r \in [0,t]}$ is the associated diffusion process and $G^1(s,x;r,z)$ is the {\em fundamental solution} of~\eqref{eq:pbdg}. Following Friedman~\cite[p.~24]{friedman}, there exists some constant $\kappa > 0$ depending on the coefficients of~\eqref{eq:pbdg} (therefore, on $\epsilon$, $\delta$ and $\eta$) such that, for all $s<r$, 
  \begin{eqnarray}
    \label{eq:G1}& &|G^1(s,x;r,z)| \leq \frac{\kappa}{(r-s)^{1/2}}\exp\left(-\frac{(z-x)^2}{\kappa(r-s)}\right),\\
    \label{eq:dG1}& &|\partial_x G^1(s,x;r,z)| \leq \frac{\kappa}{r-s}\exp\left(-\frac{(z-x)^2}{\kappa(r-s)}\right).
  \end{eqnarray}
  For $x > x_+$, \eqref{eq:FK1} combined with~\eqref{eq:G1} yields
  \begin{equation*}
    \begin{split}
      |\partial_x \g(s,x)| &\leq \int_s^t\int_{z=x_-}^{x_+}||\partial_x f||_{\infty} \frac{\kappa}{(r-s)^{1/2}}\exp\left(-\frac{(z-x_+)^2}{\kappa(r-s)}\right)\dd z\dd r\\
      &\leq \kappa||\partial_x f||_{\infty} (x_+-x_-) \exp\left(-\frac{(x-x_+)^2}{\kappa(t-s)}\right) \int_s^t\frac{\dd r}{(r-s)^{1/2}}\\
      &\leq 2(t-s)^{1/2}\kappa||\partial_x f||_{\infty} (x_+-x_-) \exp\left(-\frac{(x-x_+)^2}{\kappa(t-s)}\right),
    \end{split}  
  \end{equation*}
  and similar arguments for $x < x_-$ lead to~\eqref{eq:dgL1}. Likewise, for $x > x_+$, using~\eqref{eq:dG1} one gets
  \begin{equation*}
    \begin{split}
      |\partial_x^2 \g(s,x)| &\leq \int_s^t\int_{z=x_-}^{x_+} ||\partial_x f||_{\infty}\frac{\kappa}{r-s}\exp\left(-\frac{(z-x_+)^2}{\kappa(r-s)}\right) \dd z\dd r\\
      &\leq \kappa||\partial_x f||_{\infty}(x_+-x_-)\int_s^t \frac{1}{r-s}\exp\left(-\frac{(x-x_+)^2}{\kappa(r-s)}\right) \dd r.
    \end{split}
  \end{equation*}
  Writing, thanks to the change of variable $v = (x-x_+)/(\kappa(r-s))^{1/2}$
  \begin{equation*}
    \int_s^t \frac{1}{r-s}\exp\left(-\frac{(x-x_+)^2}{\kappa(r-s)}\right) \dd r  = \int_{(x-x_+)/(\kappa(t-s))^{1/2}}^{+\infty} \frac{2}{v}\exp(-v^2)\dd v,
  \end{equation*}
  and using the fact that, as soon as $x \geq x_+ + (t\kappa)^{1/2}$,
  \begin{equation*}
    \forall v \geq \frac{x-x_+}{(\kappa(t-s))^{1/2}}, \qquad \frac{1}{v} \leq \frac{\kappa(t-s)}{(x-x_+)^2} v,
  \end{equation*}
  we deduce that for $x \geq x_+ + (t\kappa)^{1/2}$,
  \begin{equation*}
    \int_{(x-x_+)/(\kappa(t-s))^{1/2}}^{+\infty} \frac{2}{v}\exp(-v^2)\dd v \leq \frac{\kappa(t-s)}{(x-x_+)^2} \exp\left(-\frac{(x-x_+)^2}{\kappa(t-s)}\right) \leq \exp\left(-\frac{(x-x_+)^2}{\kappa(t-s)}\right).
  \end{equation*}
  By similar arguments for $x < x_-$, one finally concludes to~\eqref{eq:d2gL1}.
\end{proof}

By the definition of $\g$,
\begin{equation}\label{eq:gsol1}
  \begin{split}
    & \int_{[0,+\infty) \times \R} (F^2_s(x)-F^1_s(x))f(s,x) \dd s\dd x = \int_{Q_t} (F^2_s(x)-F^1_s(x))f(s,x) \dd s\dd x \\
    & \qquad = \int_{Q_t} (F^2_s-F^1_s)\left\{\frac{1}{2}(\eta + \tA_{\epsilon})\partial^2_x \g + \tB_{\eta,\epsilon}^{\delta}\partial_x \g + \partial_s \g\right\}\dd s\dd x.
  \end{split}
\end{equation}

It follows from the boundedness of $\tA_{\epsilon}$ and $\tB^{\delta}_{\eta,\epsilon}$ and from Lemma~\ref{lem:g} that 
\begin{equation*}
  \sup_{s \in [0,t]} |\partial_s \g(s,x)| \leq \kappa(\epsilon,\delta,\eta)\exp(-\kappa(\epsilon,\delta,\eta)x^2).
\end{equation*}
Consequently, $\g \in C^{1,2}_{\mathrm{b}}([0,t]\times\R)$, therefore due to remark~\ref{rk:C12},
\begin{equation}\label{eq:gsol2}
  \int_{Q_t} (F^2_s-F^1_s) \left\{\frac{1}{2} \tA\partial^2_x \g + \tB\partial_x \g + \partial_s \g\right\} \dd s\dd x = 0.
\end{equation}
As a conclusion, subtracting~\eqref{eq:gsol2} to~\eqref{eq:gsol1},
\begin{equation}\label{eq:int}
  \begin{split}
    & \int_{[0,+\infty) \times \R} (F^2_s(x)-F^1_s(x))f(s,x) \dd s\dd x\\
    & \qquad = \int_{Q_t} (F^2_s-F^1_s)\left\{\frac{1}{2}(\eta + \tA_{\epsilon} - \tA)\partial^2_x \g + (\tB_{\eta,\epsilon}^{\delta} - \tB)\partial_x \g\right\}\dd s\dd x.
  \end{split}
\end{equation}

We now have to prove that the right-hand side of~\eqref{eq:int} goes to $0$ as $\epsilon, \delta, \eta \to 0$. In this purpose, we closely follow the line of~\cite{liu}. In particular, the proofs of our Lemmas~\ref{lem:1} and~\ref{lem:3} are nothing but transcriptions of the proofs of Lemmas~1 and~3 in~\cite{liu} to the framework of an unbounded domain $Q_t$ and weak solutions in $\F(T)$. Then the estimates \eqref{eq:supg}--\eqref{eq:d2gL1} ensure that the computations still make sense.

\sk
Recall that $C$ refers to a positive constant that does not depend on $\epsilon$, $\eta$ or $\delta$.

\begin{lem}\cite[Lemma~1]{liu}\label{lem:1}
  The functions $\partial_x \g$ and $\partial_x^2 \g$ are such that:
  \begin{eqnarray}
    \label{eq:lem1:i} & & \int_{Q_t} \frac{1}{2}(\eta+\tA_{\epsilon})(\partial_x^2 \g)^2 \dd s\dd x \leq \frac{K(\delta)}{\eta} + C,\\
    \label{eq:lem1:ii} & & \int_{Q_t} (\partial_x \g)^2 \dd s\dd x \leq \frac{K(\delta)}{\eta} + C.
  \end{eqnarray}
\end{lem}

\begin{lem}\cite[Lemma~3]{liu}\label{lem:3}
  The function $\partial_x \g$ is such that:
  \begin{equation}\label{eq:intdg}
    \sup_{s \in [0,t]} \int_{\R} |\partial_x \g(s,x)|\dd x \leq C.
  \end{equation}
\end{lem}

The estimates of Lemmas~\ref{lem:g}, \ref{lem:1} and~\ref{lem:3} give sufficient uniformity over the derivatives of $\g$ to conclude.
\begin{prop}\label{prop:int}
  The right-hand side of~\eqref{eq:int} is arbitrarily small when $\epsilon, \eta, \delta \to 0$.
\end{prop}
\begin{proof}
  For lighter notations, let us denote $\bF(s,x) = F^2_s(x) - F^1_s(x)$, and
  \begin{equation*}
    \begin{split}
      I &:= \int_{Q_t} \bF \frac{1}{2}(\tA_{\epsilon}-\tA)\partial^2_x \g \dd s\dd x + \int_{Q_t} \bF \frac{\eta}{2}\partial^2_x \g \dd s\dd x + \int_{Q_t} \bF (\tB_{\eta,\epsilon}^{\delta} - \tB)\partial_x \g \dd s\dd x\\
      &=: I_1 + I_2 + I_3.
    \end{split}
  \end{equation*}
  Recall that since $F^1, F^2 \in \F(T)$, then $\bF \in (L^1 \cap L^{\infty})(Q_t)$.
  
  Owing to the Cauchy-Schwarz inequality,
  \begin{equation}\label{eq:A-A}
    |I_1| \leq \left(\int_{Q_t} |\bF|\frac{(\tA_{\epsilon}-\tA)^2}{2(\eta+\tA_{\epsilon})}\dd s\dd x\right)^{1/2}\left(\int_{Q_t} |\bF|\frac{1}{2}(\eta+\tA_{\epsilon})(\partial^2_x \g)^2\dd s\dd x\right)^{1/2}.
  \end{equation}
  Using~\eqref{eq:est}, by dominated convergence the first integral in the right-hand side of~\eqref{eq:A-A} goes to $0$ as $\epsilon \to 0$ for fixed $\eta$ and $\delta$. According to Lemma~\ref{lem:1}, the second integral in~\eqref{eq:A-A} is bounded by $C+K(\delta)/\eta$. Therefore, for fixed $\eta$ and $\delta$, $\lim_{\epsilon \to 0} I_1 = 0$.
  
  Let $\alpha > 0$. Recalling the definition~\eqref{eq:FG} of $F_{\alpha}$ and $G_{\alpha}$, let us write
  \begin{equation*}
    |I_2| \leq \frac{\eta}{2}\int_{F_{\alpha}} |\bF\partial^2_x \g| \dd s\dd x +\frac{\eta}{2}\int_{G_{\alpha}} |\bF\partial^2_x \g| \dd s\dd x.
  \end{equation*}
  By the Cauchy-Schwarz inequality,
  \begin{equation*}
    \begin{split}
      \int_{F_{\alpha}} |\bF\partial^2_x \g| \dd s\dd x & \leq \left(\int_{F_{\alpha}} |\bF| \frac{\dd s\dd x}{(1/2)(\eta+\tA_{\epsilon})}\right)^{1/2}\left(\int_{F_{\alpha}} |\bF| \frac{1}{2}(\eta+\tA_{\epsilon})(\partial_x^2 \g)^2\dd s\dd x\right)^{1/2}\\
      & \leq C \left(\sup_{F_{\alpha}} \frac{1}{(1/2)(\eta+\tA_{\epsilon})}\right)^{1/2}\left(C+\frac{K(\delta)}{\eta}\right)^{1/2}\\
      & \leq \frac{C}{L(\alpha)}\left(C+\frac{K(\delta)}{\eta}\right)^{1/2},
    \end{split}
  \end{equation*}
  where $L(\alpha)$ only depends on $\alpha$. Likewise,
  \begin{equation*}
    \begin{split}
      \int_{G_{\alpha}} |\bF\partial^2_x \g| \dd s\dd x & \leq \left(\int_{G_{\alpha}} |\bF| \frac{\dd s\dd x}{(1/2)(\eta+\tA_{\epsilon})}\right)^{1/2}\left(\int_{G_{\alpha}} |\bF| \frac{1}{2}(\eta+\tA_{\epsilon})(\partial_x^2 \g)^2\dd s\dd x\right)^{1/2}\\
      &\leq \frac{C}{\eta^{1/2}} \left(\alpha\left(C+\frac{K(\delta)}{\eta}\right)\right)^{1/2};
    \end{split}    
  \end{equation*}
  so that for fixed $\eta$, $\alpha$,
  \begin{equation*}
    \limsup_{\epsilon \to 0} |I_2| \leq C\eta\left(C+\frac{K(\delta)}{\eta}\right)^{1/2}\left(\frac{1}{L(\alpha)} + \left(\frac{\alpha}{\eta}\right)^{1/2}\right).
  \end{equation*}
  
  Finally, let us write
  \begin{equation*}
    |I_3| \leq \int_{F_{\delta}} |\bF(\tB_{\eta,\epsilon}^{\delta} - \tB)\partial_x \g|\dd s\dd x + \int_{G_{\delta}} |\bF(\tB_{\eta,\epsilon}^{\delta} - \tB)\partial_x \g|\dd s\dd x. 
  \end{equation*}
  We first deal with the integral on $G_{\delta}$. On account of Lemma~\ref{lem:1}, for given $\delta$ and $\eta$ the family $(|\partial_x\g|)_{\epsilon > 0}$ is bounded in $L^2(Q_t)$. Therefore there exists a sequence $(\epsilon_k)_{k \geq 1}$ decreasing to $0$, such that $|\partial_x\gk|$ converges weakly in $L^2(Q_t)$ to a function $h \geq 0$ when $k \to +\infty$. From now on, the convergence $\epsilon \to 0$ will always be understood along the sequence $(\epsilon_k)_{k \geq 1}$. According to Lemma~\ref{lem:3}, for all compact subset $D \subset Q_t$, 
  \begin{equation*}
    \int_D h\dd s\dd x = \lim_{\epsilon \to 0} \int_D |\partial_x\g|\dd s\dd x \leq C
  \end{equation*}
  so that 
  \begin{equation*}
    \int_{Q_t} h\dd s\dd x \leq C.
  \end{equation*}
  Furthermore, on $G_\delta$ one has $\tB^{\delta}_{\eta,\epsilon} \to 0$ a.e. when $\epsilon \to 0$, and $||\tB^{\delta}_{\eta,\epsilon}||_{\infty} \leq K(\delta)$. Since $\bF \in (L^1\cap L^{\infty})(Q_t) \subset L^2(Q_t)$, by dominated convergence one deduces that $\mathbf{1}_{G_{\delta}}|\bF(\tB_{\eta,\epsilon}^{\delta} - \tB)|$ converges strongly in $L^2(Q_t)$ to $\mathbf{1}_{G_{\delta}}|\bF \tB|$. Finally,
  \begin{equation*}
    \lim_{\epsilon \to 0} \int_{G_{\delta}} |\bF(\tB_{\eta,\epsilon}^{\delta} - \tB)\partial_x \g|\dd s\dd x = \int_{G_{\delta}} |\bF \tB|h\dd s\dd x \leq C\delta.
  \end{equation*}
  
  We now turn to the integral on $F_{\delta}$. By the Cauchy-Schwarz inequality,
  \begin{equation*}
    \int_{F_{\delta}} |\bF(\tB_{\eta,\epsilon}^{\delta} - \tB)\partial_x \g|\dd s\dd x \leq \left(\int_{F_{\delta}} |\bF|(\tB_{\eta,\epsilon}^{\delta} - \tB)^2\dd s\dd x\right)^{1/2}\left(\int_{F_{\delta}} |\bF|(\partial_x \g)^2\dd s\dd x\right)^{1/2}.
  \end{equation*}
  Owing to Lemma~\ref{lem:1} and the boundedness of $\bF$,
  \begin{equation*}
    \int_{F_{\delta}} |\bF|(\partial_x \g)^2\dd s\dd x \leq C + \frac{K(\delta)}{\eta}
  \end{equation*}
  and by construction, $\tB_{\eta,\epsilon}^{\delta} \to \tB$ a.e. in $F_{\delta}$ when $\epsilon \to 0$, while $||\tB_{\eta,\epsilon}^{\delta}||_{\infty} \leq K(\delta)$. By dominated convergence, on concludes that for fixed $\eta$ and $\delta$, $\limsup_{\epsilon \to 0} |I_3| \leq C\delta$.
  
  Combining the previous estimates, let us now write
  \begin{equation*}
    \limsup_{\epsilon \to 0} |I| \leq C\eta\left(C+\frac{K(\delta)}{\eta}\right)^{1/2}\left(\frac{1}{L(\alpha)} + \left(\frac{\alpha}{\eta}\right)^{1/2}\right) + C\delta
  \end{equation*}
  and conclude by taking consecutively $\eta \to 0$, $\alpha \to 0$ and $\delta \to 0$.
\end{proof}

It follows from Proposition~\ref{prop:int} and~\eqref{eq:int} that
\begin{equation*}
  \int_{[0,T) \times \R} (F^2_s(x)-F^1_s(x))f(s,x) \dd s\dd x = 0.
\end{equation*}
Since $f$ is arbitrary, $F^1_s(x) = F^2_s(x)$ a.e. in $Q_t$, and this holds for all $t \in [0,T)$. As a consequence, $F^1 = F^2$ in $\F(T)$.

\section{Proof of Proposition~\ref{prop:quant}}\label{app:quant}

This appendix is dedicated to the computation of the time derivative of the flow $t \mapsto W_p^p(F_t,G_t)$ of the Wasserstein distance between two solutions $F$ and $G$ of the Cauchy problem~\eqref{eq:myPDE} with respective initial conditions $F_0$ and $G_0$ and classical regularity. In Subsection~\ref{app:ss:tail}, we gather some tail estimates on the solutions $F_t$ and $G_t$ as well as their space derivatives. In Subsection~\ref{app:ss:quant}, we use a new expression of $W_p^p(F_t,G_t)$ in terms of $F_t$ and $G_t$ to prove Proposition~\ref{prop:quant}.

\subsection{Tail estimates}\label{app:ss:tail} In this subsection we are concerned with the asymptotic behaviour of $F_t(x)$ and $\partial_x F_t(x)$ when $|x|$ is large.

\begin{lem}\label{lem:FtF0}
  Under the assumptions of Proposition~\ref{prop:sp}, for all $t > 0$, there exists a finite constant $C(t) > 0$ such that the function $t \mapsto C(t)$ is nondecreasing and:
  \begin{equation*}
    \begin{aligned}
      & \forall x \leq -C(t), \qquad  \frac{1}{2} F_0(2x-C(t)) \leq F_t(x) \leq F_0\left(\frac{x}{2} + C(t)\right) + \exp\left(-\frac{x^2}{C(t)}\right),\\
      & \forall x \geq C(t), \qquad \frac{1}{2}[1-F_0(2x+C(t))] \leq 1-F_t(x) \leq 1-F_0\left(\frac{x}{2}-C(t)\right) + \exp\left(-\frac{x^2}{C(t)}\right).
    \end{aligned}
  \end{equation*}
\end{lem}
\begin{proof}
  Fix a finite $T>0$. We will use the process $(\bar{X}_t)_{t \in [0,T]}$ introduced in the proof of Lemma~\ref{lem:existnlmp}. Let $t \in [0,T]$ and define $C_1(t) := t||\bar{b}||_{\infty}$. If $x \leq -C_1(t)$, then
  \begin{equation*}
    \begin{split}
      F_t(x) & = \Pr\left(\bar{X}_0 + \int_0^t \bar{b}(r,\bar{X}_r)\dd r + \int_0^t \bar{\sigma}(r,\bar{X}_r)\dd \bar{W}_r \leq x\right) \\
      & \geq \Pr\left(\bar{X}_0 + \int_0^t \bar{\sigma}(r,\bar{X}_r)\dd \bar{W}_r \leq x - t||\bar{b}||_{\infty}\right) \\
      & \geq \Pr\left(\bar{X}_0 \leq 2(x - t||\bar{b}||_{\infty}), \quad \left|\int_0^t \bar{\sigma}(r,\bar{X}_r)\dd \bar{W}_r\right| \leq t||\bar{b}||_{\infty}-x \right)\\
      & = \Pr\left(\left|\int_0^t \bar{\sigma}(r,\bar{X}_r)\dd \bar{W}_r\right| \leq t||\bar{b}||_{\infty}-x \quad\Big|\quad \bar{X}_0 \leq 2x - C_2(t)\right)\Pr(\bar{X}_0 \leq 2x - C_2(t)),
    \end{split}
  \end{equation*}
  where $C_2(t) := 2C_1(t) = 2t||\bar{b}||_{\infty}$. By Chebyshev's inequality,
  \begin{equation*}
    \begin{split}
      & \Pr\left(\left|\int_0^t \bar{\sigma}(r,\bar{X}_r)\dd \bar{W}_r\right| \leq t||\bar{b}||_{\infty}-x \quad\Big|\quad \bar{X}_0 \leq 2x - C_2(t)\right)\\
      & \qquad\geq 1-\frac{\Exp\left(\left|\displaystyle\int_0^t \bar{\sigma}(r,\bar{X}_r)\dd \bar{W}_r\right|^2\quad\Big|\quad \bar{X}_0 \leq 2x - C_2(t)\right)}{(t||\bar{b}||_{\infty}-x)^2}\\
      & \qquad\geq 1-\frac{t||\bar{\sigma}^2||_{\infty}}{(t||\bar{b}||_{\infty}-x)^2},
    \end{split}
  \end{equation*}
  and the right-hand side is larger than $1/2$ as soon as $x \leq -C_3(t) := -(2t||\bar{\sigma}^2||_{\infty})^{1/2}$.
  
  As far as the upper bound is concerned, for $x \leq -C_1(t)$,
  \begin{equation*}
    \begin{split}
      F_t(x) & = \Pr\left(\bar{X}_0 + \int_0^t \bar{b}(r,\bar{X}_r)\dd r + \int_0^t \bar{\sigma}(r,\bar{X}_r)\dd \bar{W}_r \leq x\right) \\
      & \leq \Pr\left(\bar{X}_0 + \int_0^t \bar{\sigma}(r,\bar{X}_r)\dd \bar{W}_r \leq x + t||\bar{b}||_{\infty}\right)\\
      & = \int_{\R} \Pr\left(\int_0^t \bar{\sigma}(r,\bar{X}_r)\dd \bar{W}_r \leq x - y + t||\bar{b}||_{\infty} \quad \Big| \quad \bar{X}_0 = y\right)m(\dd y).
    \end{split}
  \end{equation*}
  Let us fix $x \leq -C_1(t)$. For $x_0 \in \R$, let us split the integral in the right-hand side above in two parts, integrating respectively on $(-\infty, x_0]$ and $(x_0, +\infty)$. Then the first part can be bounded by $F_0(x_0)$, whereas for the second part the exponential Markov inequality yields, for all $\lambda > 0$,
  \begin{equation*}
    \Pr\left(\int_0^t \bar{\sigma}(r,\bar{X}_r)\dd \bar{W}_r \leq x - y + t||\bar{b}||_{\infty} \quad \Big| \quad \bar{X}_0 = y\right) \leq \exp\left(\lambda(x-y+t||\bar{b}||_{\infty}) + t\lambda^2\frac{||\bar{\sigma}^2||_{\infty}}{2}\right);
  \end{equation*}
  and finally, $F_t(x) \leq F_0(x_0) + \exp\left(\lambda(x-x_0+t||\bar{b}||_{\infty}) + t\lambda^2||\bar{\sigma}^2||_{\infty}/2\right)$. As soon as $x_0 > x + t||\bar{b}||_{\infty}$, optimizing this expression in $\lambda>0$ yields
  \begin{equation*}
    F_t(x) \leq F_0(x_0) + \exp\left(-\frac{(x-x_0+t||\bar{b}||_{\infty})^2}{2t||\bar{\sigma}^2||_{\infty}}\right).
  \end{equation*}
  We now choose $x_0 = (x + t||\bar{b}||_{\infty})/2$, then $x_0 < 0$ and $F_0(x_0) = F_0((x + t||\bar{b}||_{\infty})/2) = F_0(x/2 + C_4(t))$ with $C_4(t) := t||\bar{b}||_{\infty}/2$. Moreover, for $x \leq -C_2(t) = -2t||\bar{b}||_{\infty}$,
  \begin{equation*}
    \exp\left(-\frac{(x-x_0+t||\bar{b}||_{\infty})^2}{2t||\bar{\sigma}^2||_{\infty}}\right) = \exp\left(-x^2\frac{(1+t||\bar{b}||_{\infty}/x)^2}{8t||\bar{\sigma}^2||_{\infty}}\right) \leq \exp\left(-\frac{x^2}{C_5(t)}\right),
  \end{equation*}
  where $C_5(t) := 32 t ||\bar{\sigma}^2||_{\infty}$. We get the first part of the lemma by taking $C(t)$ as the maximum of $C_1(t), \ldots, C_5(t)$, and the second part follows similarly.
\end{proof}

Assuming the classical regularity of $F$, we now derive some estimates on the probability density $\partial_x F_t$ from the celebrated Aronson inequalities on the fundamental solution of a parabolic equation with divergence form operator~\cite{aronson}. Due to the possible dispersion of the initial condition, our upper bound contains an extra tail term in addition to the classical Gaussian term.

\begin{lem}\label{lem:aronson}
  Under the assumptions of Lemma~\ref{lem:EDPF-1}, for all $0 < t_1 < t_2$, there exists a positive constant $K>0$, depending on $t_1$ and $t_2$, such that for all $(t,x) \in (t_1,t_2]\times\R$,
  \begin{equation}\label{eq:aronson}
    \frac{1}{K(t-t_1)^{1/2}}\exp\left(-\frac{Kx^2}{t-t_1}\right) \leq \partial_x F_t (x) \leq \frac{K}{(t-t_1)^{1/2}}\left(\exp\left(-\frac{x^2}{K(t-t_1)}\right) + \ta\left(F_{t_1},\frac{x}{2}\right)\right).
  \end{equation}
\end{lem}

\begin{proof}
  The assumptions of Lemma~\ref{lem:EDPF-1} together with Corollary~\ref{cor:nlmp} ensure that the nonlinear martingale problem of Subsection~\ref{ss:nlmp} has a unique weak solution $P$. We denote by $X$ the associated nonlinear diffusion process. Let $\Gamma(s,y;t,\dd x) = \Pr(X_t \in \dd x | X_s = y)$ be the transition probability of $X$. The generator 
  \begin{equation*}
    L f := \frac{1}{2}a(F_t(x)) \partial^2_x f + b(F_t(x)) \partial_x f
  \end{equation*}
  is uniformly elliptic, and by the regularity assumptions on $A$, $B$ and $F$, it rewrites in the divergence form 
  \begin{equation*}
    L f = \frac{1}{2}\partial_x\big(a(F_t(x))\partial_x f\big) - \left(\frac{1}{2}a'(F_t(x))\partial_x F_t(x) - b(F_t(x))\right)\partial_x f. 
  \end{equation*}
  Let $0 < t_1 < t_2$. The assumption of boundedness of $\partial_x F_t(x)$ on $[t_1,t_2]\times\R$ ensures that the coefficients of the latter form are bounded. Thus, owing to~\cite{aronson}, $\Gamma(s,y;t,\dd x)$ admits a density $g(s,y;t,x)$ and there exist some positive constants $\gamma_i, \kappa_i$, $i \in \{1,2\}$, depending on $t_1$ and $t_2$, such that for all $t \in (t_1,t_2]$, 
  \begin{equation*}
    \frac{\kappa_1}{(t-t_1)^{1/2}} \exp\left(-\frac{(x-y)^2}{\gamma_1(t-t_1)}\right) \leq g(t_1,y;t,x) \leq \frac{\kappa_2}{(t-t_1)^{1/2}} \exp\left(-\frac{(x-y)^2}{\gamma_2(t-t_1)}\right).
  \end{equation*}
  Hence, $p_t^1(x) \leq \partial_x F_t(x) \leq p_t^2(x)$, where
  \begin{equation*}
    p_t^i(x) := \frac{\kappa_i}{(t-t_1)^{1/2}} \int_{\R} \exp\left(-\frac{(x-y)^2}{\gamma_i(t-t_1)}\right)P_{t_1}(\dd y).
  \end{equation*}
    
  For all $x \geq 0$,
  \begin{equation*}
    \int_{-\infty}^{x/2} \exp\left(-\frac{(x-y)^2}{\gamma_2 (t-t_1)}\right)P_{t_1}(\dd y) \leq \exp\left(-\frac{x^2}{4\gamma_2 (t-t_1)}\right)
  \end{equation*}
  while
  \begin{equation*}
    \int_{x/2}^{+\infty} \exp\left(-\frac{(x-y)^2}{\gamma_2 (t-t_1)}\right)P_{t_1}(\dd y) \leq 1-F_{t_1}(x/2).
  \end{equation*}
  Likewise, for $x \leq 0$,
  \begin{equation*}
    \int_{x/2}^{+\infty} \exp\left(-\frac{(x-y)^2}{\gamma_2 (t-t_1)}\right)P_{t_1}(\dd y) \leq \exp\left(-\frac{x^2}{4\gamma_2 (t-t_1)}\right)
  \end{equation*}
  while
  \begin{equation*}
    \int_{-\infty}^{x/2} \exp\left(-\frac{(x-y)^2}{\gamma_2 (t-t_1)}\right)P_{t_1}(\dd y) \leq F_{t_1}(x/2),
  \end{equation*}
  so that the upper bound of~\eqref{eq:aronson} holds for any $K \geq \kappa_2 \vee (4\gamma_2)$.
  
  As far as the lower bound is concerned, there exist $x_- < 0 < x_+$ such that $F_{t_1}(x_+) - F_{t_1}(x_-) \geq 1/2$. Then for all $x \in \R$,
  \begin{equation*}
    \begin{split}
      p^1_t(x) & \geq \frac{\kappa_1}{(t-t_1)^{1/2}}\int_{x_-}^{x_+} \exp\left(-\frac{(x-y)^2}{\gamma_1 (t-t_1)}\right)P_{t_1}(\dd y)\\
      & \geq \frac{\kappa_1}{2(t-t_1)^{1/2}} \exp\left(-\frac{((x_+\vee-x_-)+|x|)^2}{\gamma_1 (t-t_1)}\right).
    \end{split}
  \end{equation*}
  Now there exists $K > 0$ large enough, depending on $\kappa_1$, $\gamma_1$ and $x_+\vee-x_-$, such that for all $x \in \R$,
  \begin{equation*}
    \frac{\kappa_1}{2}\exp\left(-\frac{((x_+\vee-x_-)+|x|)^2}{\gamma_1 (t-t_1)}\right) \geq \frac{1}{K}\exp\left(-\frac{Kx^2}{t-t_1}\right),
  \end{equation*}
  which results in the lower bound of~\eqref{eq:aronson}.
\end{proof}

\subsection{Proof of Proposition~\ref{prop:quant}}\label{app:ss:quant} We first give a general formula for the Wasserstein distance $W_p(F,G)$ between two cumulative distribution functions $F$ and $G$.

\begin{lem}\label{lem:Wpp}
  Let $F$ and $G$ be two cumulative distribution functions on $\R$. Then, for all $p > 1$,
  \begin{equation}\label{eq:WR2}
    W_p^p(F,G) = p(p-1)\int_{\R^2} \ind{x<y}\left([G(x)-F(y)]^+ + [F(x)-G(y)]^+\right)(y-x)^{p-2}\dd x\dd y.
  \end{equation}
\end{lem}
\begin{proof} 
  Let us split the right-hand side of~\eqref{eq:WR2} into two symmetric integrals in $F$ and $G$. Thanks to the Fubini-Tonelli theorem, the first integral writes:
  \begin{equation*}
    \begin{aligned}
      & \int_{\R^2} \ind{x<y}[G(x)-F(y)]^+p(p-1)(y-x)^{p-2}\dd x\dd y\\
      & \qquad = \int_{\R^2} \ind{x<y; G(x) \geq F(y)} \left(\int_0^1 \ind{F(y) < u \leq G(x)}\dd u\right) p(p-1)(y-x)^{p-2}\dd x\dd y\\
      & \qquad = \int_0^1 \int_{\R^2} \ind{x<y; F(y) < u \leq G(x)} p(p-1)(y-x)^{p-2}\dd x\dd y \dd u.
    \end{aligned}
  \end{equation*}
  
  By the definition of the pseudo-inverse functions $F^{-1}$ and $G^{-1}$, note that for all $x,y \in \R$ and $u \in (0,1)$, $F(y) < u$ if and only if $y < F^{-1}(u)$ and $G(x) \geq u$ if and only if $x \geq G^{-1}(u)$. Thus, the right-hand side above rewrites
  \begin{equation*}
    \begin{aligned}  
      & \int_0^1 \int_{\R^2} \ind{G^{-1}(u) \leq x < y < F^{-1}(u)} p(p-1)(y-x)^{p-2}\dd x\dd y \dd u\\
      & \qquad = \int_0^1 \ind{G^{-1}(u) < F^{-1}(u)}\int_{x=G^{-1}(u)}^{F^{-1}(u)} \int_{y=x}^{F^{-1}(u)} p(p-1)(y-x)^{p-2}\dd y \dd x \dd u\\
      & \qquad = \int_0^1 \ind{G^{-1}(u) < F^{-1}(u)}(F^{-1}(u)-G^{-1}(u))^p \dd u,
    \end{aligned}
  \end{equation*}
  and we conclude using the symmetry in $F$ and $G$ of the two integrals in the right-hand side of~\eqref{eq:WR2}.
\end{proof}

We are now ready to complete the proof of Proposition~\ref{prop:quant}.

\begin{proof}[Proof of Proposition~\ref{prop:quant}]
  For all $t>0$,~\eqref{eq:WR2} yields
  \begin{equation}\label{eq:WppI}
    W_p^p(F_t,G_t) = p(p-1)(I(F_t,G_t) + I(G_t,F_t)),
  \end{equation}
  where we define $I(F_t,G_t) := \int_{\R^2} \ind{x<y}[G_t(x)-F_t(y)]^+(y-x)^{p-2}\dd x\dd y$. The assumption that $W_p(F_0,G_0) < +\infty$ combined with Proposition~\ref{prop:contr} ensures that both $I(F_t,G_t)$ and $I(G_t,F_t)$ are finite.
  
  For all $M \geq 0$, let us denote
  \begin{equation*}
      I_M(F_t,G_t) := \int_{\R^2} \ind{-M \leq x<y \leq M}[G_t(x)-F_t(y)]^+(y-x)^{p-2}\dd x\dd y,
  \end{equation*}
  then by the monotone convergence theorem, $\lim_{M \to +\infty} I_M(F_t,G_t) = I(F_t,G_t) < +\infty$. Owing to the assumption of classical regularity on $F$ and $G$, the function $t \mapsto I_M(F_t,G_t)$ is $C^1$ on $(0,+\infty)$ and for all $t > 0$,
  \begin{equation}\label{eq:dIMdt}
    \begin{aligned}
      & \frac{\dd}{\dd t} I_M(F_t,G_t) =\int_{\R^2} \ind{-M \leq x < y \leq M}\ind{G_t(x) \geq F_t(y)}(\partial_t G_t(x) - \partial_t F_t(y))(y-x)^{p-2}\dd x\dd y\\
      & \quad = \int_{\R^2} \ind{-M \leq x < y \leq M; G_t(x) \geq F_t(y)}\partial_x\left(\frac{1}{2}a(G_t(x))\partial_xG_t(x) - B(G_t(x))\right)(y-x)^{p-2}\dd x\dd y\\
      & \quad \quad - \int_{\R^2} \ind{-M \leq x < y \leq M; G_t(x) \geq F_t(y)}\partial_x\left(\frac{1}{2}a(F_t(y))\partial_xF_t(y) - B(F_t(y))\right)(y-x)^{p-2}\dd x\dd y.\\
    \end{aligned}
  \end{equation}
  
  Let us define $\varphi_M^+(x) := M \wedge F_t^{-1}(G_t(x))$ and $\varphi_M^-(x) := (-M) \vee G_t^{-1}(F_t(x))$. Then the first integral in the right-hand side of~\eqref{eq:dIMdt} rewrites
  \begin{equation*}
    \int_{y=-M}^M \ind{F_t(y) \leq G_t(y)} \int_{x=\varphi_M^-(y)}^y (y-x)^{p-2}\partial_x\left(\frac{1}{2}a(G_t(x))\partial_xG_t(x) - B(G_t(x))\right)\dd x \dd y
  \end{equation*}
  and integrating by parts, we get
  \begin{equation*}
    \begin{aligned}
      & \int_{x=\varphi_M^-(y)}^y (y-x)^{p-2}\partial_x\left(\frac{1}{2}a(G_t(x))\partial_xG_t(x) - B(G_t(x))\right)\dd x\\
      & \qquad = -(y-\varphi_M^-(y))^{p-2} \left(\frac{1}{2}a(G_t(-M)\vee F_t(y))\partial_xG_t(\varphi_M^-(y)) - B(G_t(-M)\vee F_t(y))\right)\\
      & \qquad \quad + \int_{x=\varphi_M^-(y)}^y (p-2)(y-x)^{p-3}\left(\frac{1}{2}a(G_t(x))\partial_xG_t(x) - B(G_t(x))\right)\dd x.
    \end{aligned}  
  \end{equation*}
  Now
  \begin{equation*}
    \begin{aligned}
      & \int_{y=-M}^M \ind{F_t(y) \leq G_t(y)} \int_{x=\varphi_M^-(y)}^y (p-2)(y-x)^{p-3}\left(\frac{1}{2}a(G_t(x))\partial_xG_t(x) - B(G_t(x))\right)\dd x\dd y\\
      & \qquad = \int_{x=-M}^M \ind{F_t(x) \leq G_t(x)}\left(\frac{1}{2}a(G_t(x))\partial_xG_t(x) - B(G_t(x))\right) \int_{y=x}^{\varphi_M^+(x)} (p-2)(y-x)^{p-3} \dd y\dd x\\
      & \qquad = \int_{x=-M}^M \ind{F_t(x) \leq G_t(x)}\left(\frac{1}{2}a(G_t(x))\partial_xG_t(x) - B(G_t(x))\right) (\varphi_M^+(x)-x)^{p-2} \dd y\dd x,
    \end{aligned}  
  \end{equation*}
  so that the first integral in the right-hand side of~\eqref{eq:dIMdt} finally writes
  \begin{equation*}
    \begin{aligned}
      & \int_{-M}^M \ind{F_t(x) \leq G_t(x)} \Big\{ (\varphi_M^+(x)-x)^{p-2}\Big(\frac{1}{2}a(G_t(x))\partial_xG_t(x) - B(G_t(x))\Big)\\
      & - (x-\varphi_M^-(x))^{p-2}\Big(\frac{1}{2}a(G_t(-M)\vee F_t(x))\partial_xG_t(\varphi_M^-(x)) - B(G_t(-M)\vee F_t(x))\Big) \Big\} \dd x
    \end{aligned}
  \end{equation*}
  whereas the second integral similarly writes
  \begin{equation*}
    \begin{aligned}
      & \int_{-M}^M \ind{F_t(x) \leq G_t(x)} \Big\{ (\varphi_M^+(x)-x)^{p-2} \Big(\frac{1}{2}a(F_t(M)\wedge G_t(x))\partial_xF_t(\varphi_M^+(x)) - B(F_t(M)\wedge G_t(x))\Big)\\
      & - (x-\varphi_M^-(x))^{p-2}\Big(\frac{1}{2}a(F_t(x))\partial_xF_t(x) - B(F_t(x))\Big) \Big\} \dd x.
    \end{aligned}
  \end{equation*}
  
  Hence, we deduce that for all $0 < t_1 \leq t_2$,
  \begin{equation*}
    I_M(F_{t_2},G_{t_2}) - I_M(F_{t_1},G_{t_1}) = \int_{t_1}^{t_2} \frac{\dd}{\dd t} I_M(F_t,G_t)\dd t = J^1_M + J^2_M + J^3_M,
  \end{equation*}
  where
  \begin{equation*}
    \begin{split}
      J^1_M & := \int_{t_1}^{t_2}\int_{-M}^M \ind{F_t(x) \leq G_t(x)}\{ (\varphi_M^+(x) - x)^{p-2}[B(F_t(M)\wedge G_t(x)) - B(G_t(x))]\\
      & \quad + (x - \varphi_M^-(x))^{p-2}[B(G_t(-M)\vee F_t(x)) - B(F_t(x))] \} \dd x \dd t;\\
      J^2_M & := \frac{1}{2}\int_{t_1}^{t_2}\int_{-M}^M \ind{F_t(x) \leq G_t(x)}\{ (\varphi_M^+(x) - x)^{p-2}a(G_t(x))\partial_x G_t(x)\\
      & \quad + (x - \varphi_M^-(x))^{p-2}a(F_t(x))\partial_x F_t(x) \} \dd x \dd t;\\
      J^3_M & := -\frac{1}{2}\int_{t_1}^{t_2}\int_{-M}^M \ind{F_t(x) \leq G_t(x)}\{ (\varphi_M^+(x) - x)^{p-2}a(F_t(M) \wedge G_t(x))\partial_x F_t(\varphi_M^+(x))\\
      & \quad + (x - \varphi_M^-(x))^{p-2}a(G_t(-M) \vee F_t(x))\partial_x G_t(\varphi_M^-(x)) \} \dd x \dd t.
    \end{split}
  \end{equation*}
  
  \sk
  {\bf Integral term $J_M^1$.} Since $B$ is $C^1$ on $[0,1]$,
  \begin{equation*}
    \begin{aligned}
      |J^1_M| & \leq \int_{t_1}^{t_2}\int_{-M}^M \ind{F_t(x) \leq G_t(x)}||b||_{\infty}\{ (\varphi_M^+(x) - x)^{p-2}|F_t(M)\wedge G_t(x) - G_t(x)|\\
      & \quad + (x - \varphi_M^-(x))^{p-2}|G_t(-M)\vee F_t(x) - F_t(x)| \} \dd x \dd t\\
      & \leq \int_{t_1}^{t_2}\int_{-M}^M ||b||_{\infty} (2M)^{p-2} \{[G_t(x) - F_t(M)]^+ + [G_t(-M)-F_t(x)]^+\}\dd x\dd t\\
      & \leq \int_{t_1}^{t_2} ||b||_{\infty} (2M)^{p-1} \{\ta(F_t,M) + \ta(G_t,-M)\}\dd t.
    \end{aligned}
  \end{equation*}
  By Lemma~\ref{lem:FtF0}, for all $M \geq 2 C(t_2)$, for all $t \in [t_1,t_2]$,
  \begin{equation*}
    \begin{aligned}
      & \ta(F_t,M) \leq \ta(F_0, M/2-C(t_2)) + \exp(-M^2/C(t_2)),\\
      & \ta(G_t,-M) \leq \ta(G_0, -M/2+C(t_2)) + \exp(-M^2/C(t_2)),
    \end{aligned}
  \end{equation*}
  so that $|J^1_M| \to 0$ when $M \to +\infty$ due to the tail assumption on $F_0$ and $G_0$.
  
  \sk
  {\bf Integral term $J_M^2$.} By the monotone convergence theorem,
  \begin{equation*}
    \begin{split}
      \lim_{M \to +\infty} J^2_M & = \frac{1}{2}\int_{t_1}^{t_2}\int_{\R} \ind{F_t(x) \leq G_t(x)}\{ (F_t^{-1}(G_t(x)) - x)^{p-2}a(G_t(x))\partial_x G_t(x)\\
      & \quad + (x - G_t^{-1}(F_t(x)))^{p-2}a(F_t(x))\partial_x F_t(x) \} \dd x \dd t,
    \end{split}  
  \end{equation*}
  and the limit is finite as
  \begin{equation*}
    \begin{aligned}
      & \int_{t_1}^{t_2}\int_{\R} \ind{F_t(x) \leq G_t(x)} (F_t^{-1}(G_t(x)) - x)^{p-2}a(G_t(x))\partial_x G_t(x) \dd x\dd t\\
      & \qquad \leq ||a||_{\infty} \int_{t_1}^{t_2}\int_{\R}|F_t^{-1}(G_t(x)) - x|^{p-2}\partial_x G_t(x) \dd x\dd t\\
      & \qquad = ||a||_{\infty} \int_{t_1}^{t_2}W_{p-2}^{p-2}(F_t,G_t)\dd t\\
      & \qquad \leq ||a||_{\infty}(t_2-t_1)W_{p-2}^{p-2}(F_0,G_0) < +\infty,
    \end{aligned}  
  \end{equation*}
  due to Proposition~\ref{prop:contr} (we take the convention that $W_0^0(F_t,G_t) = 1$).
  
  \sk
  {\bf Integral term $J_M^3$.} Note that
  \begin{equation}\label{eq:J3}
    \begin{aligned}
      & \int_{t_1}^{t_2}\int_{-M}^M \ind{F_t(x) \leq G_t(x)}(\varphi_M^+(x) - x)^{p-2}a(F_t(M) \wedge G_t(x))\partial_x F_t(\varphi_M^+(x)) \dd x\dd t\\
      & = \int_{t_1}^{t_2}\int_{-M}^M \ind{F_t(x) \leq G_t(x) ; F_t(M) \leq G_t(x)}(M - x)^{p-2}a(F_t(M))\partial_x F_t(M) \dd x\dd t +\\
      & \int_{t_1}^{t_2}\int_{-M}^M \ind{F_t(x) \leq G_t(x) ; F_t(M) > G_t(x)}(F_t^{-1}(G_t(x)) - x)^{p-2}a(G_t(x))\partial_x F_t(F_t^{-1}(G_t(x))) \dd x\dd t.
    \end{aligned}
  \end{equation}
  According to Lemmas~\ref{lem:FtF0} and~\ref{lem:aronson}, letting $C := C(t_1/2)$, for $M \geq 4C$, the first integral in the right-hand side of~\eqref{eq:J3} is bounded by
  \begin{equation*}
    \begin{aligned}
      & ||a||_{\infty}\int_{t_1}^{t_2} (2M)^{p-1}\partial_x F_t(M) \dd t\\
      & \qquad \leq ||a||_{\infty}(2M)^{p-1}\int_{t_1}^{t_2} \frac{K}{(t-t_1/2)^{1/2}} \left(\exp\left(-\frac{M^2}{K(t-t_1/2)}\right) + \ta\left(F_{t_1/2}, \frac{M}{2}\right)\right) \dd t\\
      & \qquad \leq K||a||_{\infty}(2M)^{p-1}\frac{(t_2-t_1)}{(t_1/2)^{1/2}} \left(\exp\left(-\frac{M^2}{K(t_2-t_1/2)}\right) + \ta\left(F_0, \frac{M}{4} - C\right) + \exp\left(-\frac{M^2}{C}\right)\right),
    \end{aligned}
  \end{equation*}
  and the right-hand side of the last inequality vanishes when $M \to +\infty$, whereas the second integral in the right-hand side of~\eqref{eq:J3} converges monotonically to
  \begin{equation*}
    \int_{t_1}^{t_2}\int_{\R} \ind{F_t(x) \leq G_t(x)}(F_t^{-1}(G_t(x)) - x)^{p-2}a(G_t(x))\partial_x F_t(F_t^{-1}(G_t(x))) \dd x\dd t.
  \end{equation*}
  The second term in $J^3_M$ is similar.
  
  \sk
  {\bf Conclusion.} Taking the limit $M \to +\infty$ in the equality $I_M(F_{t_2},G_{t_2}) - I_M(F_{t_1},G_{t_1}) - (J^1_M + J^2_M) = J^3_M$ now yields
  \begin{equation*}
    \begin{aligned}
      & I(F_{t_2},G_{t_2}) - I(F_{t_1},G_{t_1}) - \frac{1}{2} \int_{t_1}^{t_2} \int_{\R} \ind{F_t(x) \leq G_t(x)}\{(F_t^{-1}(G_t(x)) - x)^{p-2}a(G_t(x))\partial_x G_t(x)\\
      & + (x - G_t^{-1}(F_t(x)))^{p-2}a(F_t(x))\partial_x F_t(x) \} \dd x \dd t\\
      & \qquad = -\frac{1}{2}\left(\int_{t_1}^{t_2}\int_{\R} \ind{F_t(x) \leq G_t(x)}(F_t^{-1}(G_t(x)) - x)^{p-2}a(G_t(x))\partial_x F_t(F_t^{-1}(G_t(x))) \dd x\dd t\right.\\
      & \qquad \quad  \left. + \int_{t_1}^{t_2}\int_{\R} \ind{F_t(x) \leq G_t(x)}(x - G_t^{-1}(F_t(x)))^{p-2}a(F_t(x))\partial_x G_t(G_t^{-1}(F_t(x))) \dd x\dd t\right).
    \end{aligned}
  \end{equation*}
  The left-hand side of the equality above is finite and the integrands in both integrals of the right-hand side are nonnegative. Hence, all the integrals involved are absolutely convergent, and we deduce
  \begin{equation*}
    \begin{aligned}
      & I(F_{t_2},G_{t_2}) - I(F_{t_1},G_{t_1})\\
      & \qquad = \frac{1}{2} \int_{t_1}^{t_2} \int_{\R} \ind{F_t(x) \leq G_t(x)} \{(F_t^{-1}(G_t(x))-x)^{p-2} \left(a(G_t(x))(\partial_x G_t(x) - \partial_x F_t(F_t^{-1}(G_t(x)))\right)\\
      & \qquad \quad + (x-G_t^{-1}(F_t(x)))^{p-2} \left(a(F_t(x))(\partial_x F_t(x) - \partial_x G_t(G_t^{-1}(F_t(x)))\right)\}\dd x\dd t.
    \end{aligned}
  \end{equation*}
  
  By symmetry and using~\eqref{eq:WppI}, we now conclude
  \begin{equation*}
    \begin{aligned}
      & W_p^p(F_{t_2},G_{t_2}) - W_p^p(F_{t_1},G_{t_1}) \\
      & \qquad = \frac{p(p-1)}{2} \int_{t_1}^{t_2}\int_{\R}\{|F_t^{-1}(G_t(x))-x|^{p-2} \left(a(G_t(x))(\partial_x G_t(x) - \partial_x F_t(F_t^{-1}(G_t(x)))\right)\\
      & \qquad \quad + |x-G_t^{-1}(F_t(x))|^{p-2} \left(a(F_t(x))(\partial_x F_t(x) - \partial_x G_t(G_t^{-1}(F_t(x)))\right)\}\dd x\dd t\\
      & \qquad = \frac{p(p-1)}{2}\int_{t_1}^{t_2}\left\{\int_0^1 |F_t^{-1}(u) - G_t^{-1}(u)|^{p-2} a(u) \left(\frac{1}{\partial_u G_t^{-1}(u)} - \frac{1}{\partial_u F_t^{-1}(u)}\right) \partial_u G_t^{-1}(u) \dd u\right.\\
      & \qquad \quad \left. + \int_0^1 |F_t^{-1}(u) - G_t^{-1}(u)|^{p-2} a(u) \left(\frac{1}{\partial_u F_t^{-1}(u)} - \frac{1}{\partial_u G_t^{-1}(u)}\right) \partial_u F_t^{-1}(u) \dd u\right\}\dd t\\
      & \qquad = -\frac{p(p-1)}{2}\int_{t_1}^{t_2}\int_0^1 a(u)|F_t^{-1}(u) - G_t^{-1}(u)|^{p-2} \frac{(\partial_u F_t^{-1}(u)-\partial_u G_t^{-1}(u))^2}{\partial_u F_t^{-1}(u)\partial_u G_t^{-1}(u)}\dd u\dd t;
    \end{aligned}
  \end{equation*}
  which completes the proof.
\end{proof}


\def\cprime{$'$}

\end{document}